\documentclass[11pt]{article}
\usepackage{graphicx}
\usepackage{epstopdf}
\usepackage{amssymb}
\usepackage{amsmath}
\usepackage{amsthm}
\usepackage{color}
\usepackage[margin=1in]{geometry}
\usepackage[hidelinks]{hyperref}
\usepackage{booktabs}

\numberwithin{equation}{section}

\newcommand{\be}{\begin{equation}}
\newcommand{\ee}{\end{equation}}
\newcommand{\ba}{\begin{array}}
\newcommand{\ea}{\end{array}}
\newcommand{\bea}{\begin{eqnarray}}
\newcommand{\eea}{\end{eqnarray}}
\newcommand{\beas}{\begin{eqnarray*}}
\newcommand{\eeas}{\end{eqnarray*}}
\newtheorem{theorem}{Theorem}[section]
\newtheorem{lemma}{Lemma}[section]
\newtheorem{proposition}{Proposition}[section]

\newtheorem{assumption}{Assumption}[section]

\title{A WKB-based fixed-grid method for capturing trait concentration in a dispersal evolution model}

\author{
Weijie Huang\thanks{School of Mathematics and Statistics, Beijing Jiaotong University, Beijing 100044, People's Republic of China; Beijing Key Laboratory of Biological Big Data and Topological Statistics, Beijing Jiaotong University, Beijing 100044, People's Republic of China {\tt (wjhuang@bjtu.edu.cn)}} \ and
Xinran Ruan\thanks{(Corresponding author) School of Mathematical Sciences, Capital Normal University, Beijing, China, 100048, People's Republic of China {\tt (xinran.ruan@cnu.edu.cn)}.}
}

\begin{document}
\date{}
\maketitle

\begin{abstract}
The evolution of dispersal traits is a fundamental topic in evolutionary ecology, 
where natural selection may drive the trait distribution toward concentration in the rare-mutation regime. 
This singular behavior poses a serious numerical difficulty, 
since direct discretizations of the population density require very fine trait grids to identify the fittest trait and to resolve the concentrated profile accurately. 
In this paper, we develop a WKB-based numerical framework for a dispersal evolution model. 
By separating the exponentially concentrated trait dependence from a smoother
amplitude variable and combining this WKB representation with dual trait-grid implementation and other specially designed techniques, the method recovers the selected trait and
the associated concentration structure accurately and efficiently on fixed trait grids. 
We establish a semi-discrete stationary fixed-grid asymptotic-preserving structure for the rare-mutation limit of the steady-state problem.
Numerical experiments compare the proposed method with
direct density discretizations and confirm its advantage in the small-mutation
regime.
\end{abstract}

{\bf Key words. } dispersal evolution, rare-mutation limit, trait concentration, WKB reformulation, Hamilton--Jacobi limit, asymptotic-preserving structure.

{\bf MSC 2020. } 65M06, 65M12, 35B25, 35F21, 92D15.

%===================================================================================================
%	Introduction
%===================================================================================================
\section{Introduction}

Adaptive dynamics studies the evolution of phenotypic traits under mutation, selection, and ecological feedback. In structured population models, natural selection is often manifested by the concentration of the population density on one or several fittest traits. In the rare-mutation regime, such concentration is usually described by a measure-valued limit in the trait variable, and the support of the limiting measure represents the selected phenotype. This viewpoint has been developed in selection-mutation models and in the Hamilton--Jacobi approach to adaptive dynamics \cite{BarlesPerthame, DiekmannJabinMischlerPerthame}. Among these problems, the evolution of dispersal is a particularly important example. In spatially heterogeneous environments, different dispersal rates may lead to different fitness advantages, and the selected strategy depends on the interaction between movement, resource distribution, and population competition. This question has been studied from several viewpoints in mathematical ecology, for example in reaction-diffusion models for slow dispersal, ideal free distributions, and spatial sorting \cite{DockeryHutsonMischaikowPernarowski1998, Hastings1983,Ronce2007}.

In this work, we consider a representative dispersal evolution model of the form 
\begin{equation}\label{eq:n_intro} 
	\varepsilon \partial_t n_\varepsilon -D(\theta)\Delta_{\mathbf{x}} n_\varepsilon -\varepsilon^2\partial_{\theta\theta} n_\varepsilon = n_\varepsilon\bigl(K(\mathbf{x})-\rho_\varepsilon(t, \mathbf{x})\bigr), 
\end{equation} 
where $\mathbf{x}\in\Omega\subset \mathbb{R}^d, \, d=1,2,3, \, \theta \in \mathbb{T}=(0,1].$
Here \(t\) denotes time, \(\mathbf{x}\) is the spatial variable, and \(\theta\) is the phenotypic trait. The unknown
\(n_\varepsilon(t,\mathbf{x},\theta)\) is the structured population density, and 
\begin{equation} \label{def:rho_P}
	\rho_\varepsilon(t,\mathbf{x}) = \int_{\mathbb T} n_\varepsilon(t,\mathbf{x},\theta)\,d\theta, \quad 
	P_\varepsilon(t,\theta) = \int_\Omega n_\varepsilon(t,\mathbf{x},\theta)\,d\mathbf{x},  
\end{equation}
denote the population density in space and the trait marginal, respectively. 
The concentration of \(P_\varepsilon\) in the rare-mutation regime is used to identify the selected trait. 
The function \(K(\mathbf{x})\) represents the local carrying capacity, while the trait-dependent diffusivity \(D(\theta)\) models the dispersal rate associated with phenotype \(\theta\). Throughout the paper, we assume that the dispersal rate is positive and periodic in the trait variable and that the carrying capacity is nonnegative and bounded \cite{PerthameSouganidis}
 \begin{equation}\label{eq:basic_assumptions} 
 	0<D_{\min}\le D(\theta)\le D_{\max}<\infty, \qquad 0\le K(\mathbf{x})\le K_{\max}. 
\end{equation} 
We impose homogeneous Neumann boundary conditions on \(\partial\Omega\) and periodic boundary conditions in the trait direction. The time-dependent model \eqref{eq:n_intro} is used here as an evolutionary relaxation toward the selection state, while the asymptotic structure of interest is motivated by the rare-mutation limit of the associated steady-state problem.

The steady-state version of this dispersal evolution model was studied by
B. Perthame and P. E. Souganidis in \cite{PerthameSouganidis}.  Their work provides a
rare-mutation description of the selected dispersal trait for a population
structured by space and by a continuous dispersal trait.  In the limit of
vanishing mutations, the population concentrates in the trait variable, while
the spatial profile remains nontrivial.  This result gives the asymptotic
background for the present paper and clarifies the limiting selection structure
that the numerical method aims to capture.

This dispersal model fits into a broader class of structured selection models with concentration in the small-mutation limit. For purely trait-structured or nonlocal selection models, Dirac concentration and constrained Hamilton--Jacobi equations have been studied in \cite{DesvillettesJabinMischlerRaoul} and \cite{LorzMirrahimiPerthame}. When space is included as an additional structuring variable, the analysis becomes more delicate. 
Related Hamilton--Jacobi limits for populations structured by space and trait were obtained in \cite{BouinMirrahimi}, and an asymptotic analysis of a selection model with space was given in \cite{MirrahimiPerthame}. 
The Perthame--Souganidis dispersal model has also motivated later studies on dispersal evolution, conditional dispersal, and the stability of Dirac concentrations \cite{HaoLamLou2019, Lam2017,LamLouPerthame2023}.

The same concentration mechanism creates a severe numerical difficulty.  Direct
discretizations of the original density must resolve a trait profile whose
width shrinks with the mutation scale.  Hence the mesh required in the trait
direction may become prohibitively fine when \(\varepsilon\) is small.  This
suggests that the numerical method should exploit the asymptotic separation
between the singular trait dependence and the smoother amplitude, rather than
resolve the concentrated density directly.

This is closely related to the idea of asymptotic-preserving schemes, whose
goal is to capture the limiting regime on discretizations that do not resolve
the small scale \(\varepsilon\) \cite{JinAP}.  In rare-mutation problems, the
natural asymptotic separation is provided by a Wentzel--Kramers--Brillouin (WKB) \cite{BenderOrszag1999WKB} or Hopf--Cole
transformation \cite{BarlesPerthame2008}.  The density is represented by an exponential phase, which
describes the concentration in the trait variable, multiplied by a smoother
amplitude.  Such a WKB viewpoint is standard in the Hamilton--Jacobi approach
to adaptive dynamics \cite{BarlesPerthame} and is also central in the
dispersal model of \cite{PerthameSouganidis}.

WKB-type transformations have been used as numerical tools in different multiscale regimes. For highly oscillatory problems, such as semiclassical Schr\"odinger equations, the exponential phase is typically imaginary or complex-valued, and the WKB transformation separates the rapid oscillations from smoother amplitude variables \cite{ArnoldAbdallahNegulescu2011,KoernerArnoldDoepfner2022}. In selection-mutation problems, by contrast, the phase is real-valued and the small parameter leads to concentration rather than oscillation. In adaptive dynamics, a WKB representation has been used to construct an asymptotic-preserving scheme for capturing Dirac concentrations on coarse, \(\varepsilon\)-independent phenotype meshes \cite{AlmeidaPerthameRuan}. Related Hopf--Cole based  ideas have been developed for kinetic reaction--transport equations in singular front propagation regimes \cite{Hivert2018}, and for kinetic BGK models in the hyperbolic limit \cite{LuoPayne2017}. These works motivate the WKB-based formulation adopted here for the dispersal evolution model.

The present paper develops a WKB-based fixed-grid numerical framework for the
dispersal evolution problem \eqref{eq:n_intro}.  Starting from the
phase-amplitude representation, we derive a semi-discrete scheme for the phase
and amplitude variables and prove the positivity of the amplitude.  We then
establish a stationary fixed-grid asymptotic-preserving structure associated with the
rare-mutation limit of the steady-state problem under a one-sided remainder
stability condition.  For computation, we propose a dual trait-grid
implementation.  The one-dimensional phase is resolved on a fine trait grid,
whereas the space-dependent amplitude is evolved on a coarser trait grid.  This
separation preserves the main advantage of the WKB formulation and reduces the
dominant cost associated with the space-dependent unknown.

The paper is organized as follows.  Section~\ref{sec:proliminary} recalls the WKB ansatz and the
steady-state rare-mutation structure that motivates the numerical formulation.
Section~\ref{sec:WKB_scheme} presents the semi-discrete WKB scheme, proves positivity of the
amplitude, and establishes the stationary fixed-grid asymptotic-preserving structure.  Section~\ref{sec:numerics}
describes the fully discrete implementation used in the computations and
presents numerical experiments comparing the WKB method with direct density
discretizations in the small-mutation regime.  The appendices collect the
auxiliary numerical Hamiltonians, fluxes, and remainder estimates used in the
analysis.  
Finally, conclusions are drawn in Section~\ref{sec:conclusion}.

%===================================================================================================
%   Preliminary results
%===================================================================================================
\section{Preliminaries}\label{sec:proliminary}
This section summarizes the continuous asymptotic structure that motivates the WKB reformulation and the subsequent numerical discretization.
\subsection{WKB ansatz and reformulated system}
For simplicity, we restrict the presentation to the one-dimensional spatial case, i.e. $x\in\Omega\subset\mathbb{R}$. 
To capture the concentration phenomenon in the trait variable, we introduce the classical WKB ansatz
\begin{equation}\label{eq:WKB}
   n_\varepsilon(t, x,\theta)
   = W_\varepsilon(t, x,\theta)\,
   \exp\!\left(\frac{u_\varepsilon(t,\theta)}{\varepsilon}\right).
\end{equation}
This decomposition separates the exponentially varying part in the trait direction from a comparatively smoother amplitude. It is therefore natural for both the continuous asymptotic analysis and the numerical design in the small-mutation regime.

Substituting \eqref{eq:WKB} into \eqref{eq:n_intro} and assigning the leading trait-direction exponential contributions to the phase function, we obtain the reformulated system
\begin{equation}\label{eq:WKB_reformulation0}
\left\{
\begin{aligned}
& \partial_t u_\varepsilon
- |\partial_\theta u_\varepsilon|^2
- \varepsilon \, \partial_{\theta\theta} u_\varepsilon
= - \mathcal H, \\
& \varepsilon \partial_t W_\varepsilon
- D(\theta)\Delta_x W_\varepsilon
- \varepsilon^2 \partial_{\theta\theta} W_\varepsilon
- 2\varepsilon \, \partial_\theta W_\varepsilon \, \partial_\theta u_\varepsilon
= W_\varepsilon\bigl(K(x)-\rho_\varepsilon(t,x)+\mathcal H \bigr). 
\end{aligned}
\right.
\end{equation}
Here \(\mathcal H\) is introduced as an effective Hamiltonian. 
As we can see later, a natural choice is given by the principal eigenvalue of an associated spatial eigenvalue problem \eqref{eq:eigen_N}. 
Using the identity
\[
\partial_\theta(W_\varepsilon\partial_\theta u_\varepsilon)
=
\partial_\theta W_\varepsilon\,\partial_\theta u_\varepsilon
+
W_\varepsilon\partial_{\theta\theta}u_\varepsilon,
\]
we rewrite the term  $-2\varepsilon\partial_\theta W_\varepsilon\,\partial_\theta u_\varepsilon$ in \eqref{eq:WKB_reformulation0} as $2\varepsilon\partial_\theta F_\varepsilon+2\varepsilon W_\varepsilon\partial_{\theta\theta}u_\varepsilon$, where $F_\varepsilon = - W_\varepsilon\partial_\theta u_\varepsilon$. 
Consequently, the system \eqref{eq:WKB_reformulation0} can also be written equivalently as
\begin{equation}\label{eq:WKB_reformulation}
\left\{
\begin{aligned}
& \partial_t u_\varepsilon
- |\partial_\theta u_\varepsilon|^2
- \varepsilon \, \partial_{\theta\theta} u_\varepsilon
= - \mathcal H, \\
& \varepsilon \partial_t W_\varepsilon
- D(\theta)\Delta_x W_\varepsilon
- \varepsilon^2 \partial_{\theta\theta} W_\varepsilon
+2\varepsilon\partial_\theta F_\varepsilon
=
W_\varepsilon
\bigl(
K(x)-\rho_\varepsilon(t, x)
+\mathcal H
-2\varepsilon\partial_{\theta\theta}u_\varepsilon
\bigr).
\end{aligned}
\right.
\end{equation}

\subsection{Continuous asymptotic structure of the steady state}\label{subsec:continuous_asymptotic_structure}

We next recall the steady-state asymptotic structure that motivates the numerical method. Let \((\bar n_\varepsilon,\bar \rho_\varepsilon)\) be a positive steady state of \eqref{eq:n_intro}. Then
\begin{equation}\label{eq:ss_n}
-D(\theta)\Delta_x \bar n_\varepsilon
-\varepsilon^2\partial_{\theta\theta} \bar n_\varepsilon
=
\bar n_\varepsilon\bigl(K(x)-\bar\rho_\varepsilon(x)\bigr),
\qquad
\bar\rho_\varepsilon(x)=\int_{\mathbb T} \bar n_\varepsilon(x,\theta)\,d\theta.
\end{equation}
As shown in \cite{PerthameSouganidis}, under standard assumptions on \(K\), \(D\), and the domain, the steady-state family admits several properties that are fundamental for the rare-mutation limit.

\begin{theorem}[Continuous asymptotic structure of the steady state]\label{prop:continuous_structure}
Assume that \(K\) is positive and nonconstant, that \(D\) is positive, periodic, and has a unique minimizer \(\theta_m\), and that the steady problem \eqref{eq:ss_n} admits a positive solution. Then the following properties hold at the continuous level.
\begin{enumerate}
\item There exists a constant \(C>0\), independent of \(\varepsilon\), such that
\begin{equation}\label{eq:rho_continuous_bound}
0\le \bar\rho_\varepsilon(x)\le C
\qquad \text{for all } x\in\Omega.
\end{equation}
Moreover, up to subsequences, \(\bar\rho_\varepsilon\) converges uniformly to a limit density \(\rho\).

\item If we write
\begin{equation}\label{eq:u_log_transform}
\bar n_\varepsilon(x,\theta)=\exp\!\left(\frac{\bar u_\varepsilon(x,\theta)}{\varepsilon}\right),
\end{equation}
then \(\bar u_\varepsilon\) is uniformly Lipschitz continuous in \(\theta\), and along subsequences
\begin{equation}\label{eq:u_limit_uniform}
\bar u_\varepsilon \to u
\qquad \text{uniformly in } (x,\theta),
\end{equation}
where the limit \(u\) is independent of \(x\), periodic in \(\theta\), and satisfies
\begin{equation}\label{eq:HJ_limit}
\begin{cases}
|\partial_\theta u(\theta)|^2 = \mathcal H\bigl(\theta,\rho(\cdot)\bigr),\\[1mm]
\displaystyle \max_{\theta\in\mathbb T} u(\theta)=0.
\end{cases}
\end{equation}

\item The effective Hamiltonian \(\mathcal H(\theta,\rho)\) is defined through the principal eigenvalue problem
\begin{equation}\label{eq:eigen_N}
\begin{cases}
-D(\theta)\Delta_x \mathcal N(x,\theta)
= \mathcal N(x,\theta)\bigl(K(x)-\rho(x)\bigr)
+ \mathcal N(x,\theta)\,\mathcal H\bigl(\theta,\rho(\cdot)\bigr), & x\in\Omega,\\
\partial_\nu \mathcal N = 0, & x\in\partial\Omega,
\end{cases}
\end{equation}
with a positive eigenfunction \(\mathcal N\), where $\nu$ denotes the outward unit normal  vector on $\partial \Omega$. Moreover, \(\mathcal H\) has the same monotonicity in \(\theta\) as \(D\), and hence
\begin{equation}\label{eq:H_min_zero}
\min_{\theta\in\mathbb T} \mathcal H\bigl(\theta,\rho(\cdot)\bigr)
=
\mathcal H\bigl(\theta_m,\rho(\cdot)\bigr)
=0.
\end{equation}

\item The steady state concentrates on the fittest trait in the limit \(\varepsilon\to0\). More precisely,
\begin{equation}\label{eq:continuous_dirac_limit}
\bar n_\varepsilon \rightharpoonup N_m(x)\,\delta(\theta-\theta_m),
\qquad
\bar\rho_\varepsilon \to N_m(x),
\end{equation}
where \(N_m\) solves the reduced Fisher-type problem
\begin{equation}\label{eq:Nm_problem}
\begin{cases}
-D(\theta_m)\Delta_x N_m = N_m\bigl(K(x)-N_m\bigr), & x\in\Omega,\\
\partial_\nu N_m = 0, & x\in\partial\Omega.
\end{cases}
\end{equation}
\end{enumerate}
\end{theorem}

Theorem~\ref{prop:continuous_structure} recalls the continuous
rare-mutation structure that motivates the numerical construction.  It
shows that the selected trait is encoded by the limiting phase, while the
spatial profile is determined through an effective Hamiltonian. 
The discrete counterpart of this structure will be discussed later in Section~\ref{sec:WKB_scheme}.

%===================================================================================================
%   Detailed scheme
%===================================================================================================

\section{WKB reformulated numerical scheme}\label{sec:WKB_scheme}

In this section, we develop and analyze a semi-discrete numerical scheme for
the WKB reformulation \eqref{eq:WKB_reformulation}. The scheme is formulated
in terms of the phase variable \(u_\varepsilon\) and the amplitude variable
\(W_\varepsilon\), with the total density \(n_\varepsilon\) reconstructed from
the WKB representation.

\subsection{Semi-discrete formulation}

We work at the semi-discrete level in order to separate the discretization in
\((x,\theta)\) from the additional issue of time stepping. For clarity, we
present the scheme in one spatial dimension with uniform grids in the
\(x\)- and \(\theta\)-directions. This choice is not essential for the WKB
decomposition, but it keeps the notation transparent.

Let \(\{x_j\}_{j=0}^{N_x}\) be a uniform grid on \(\Omega=(a,b)\) and let
\(\{\theta_k\}_{k=1}^{N_\theta}\) be a uniform periodic grid on \(\mathbb T=(0,1]\),
with mesh size $\Delta x$ and \(\Delta\theta\). We denote
\[
W_{j,k}^\varepsilon(t)\approx W_\varepsilon(t,x_j,\theta_k),
\qquad
u_k^\varepsilon(t)\approx u_\varepsilon(t,\theta_k).
\]
For notational convenience, we write
\[
D_k=D(\theta_k), \quad k = 1,2,\cdots, N_\theta,
\quad
K_j=K(x_j), \quad j=0, 1,\cdots, N_x. 
\] 
By \eqref{eq:basic_assumptions}, we have 
\[
    0<D_{\min}\le D_k\le D_{\max}<\infty,
    \qquad
    0\le K_j\le K_{\max}.
\]

For a grid function \(V=\{V_{j,k}\}\), we use the standard second-order
difference operators
\[
\delta_x^2 V_{j,k}
=
\frac{V_{j+1,k}-2V_{j,k}+V_{j-1,k}}{(\Delta x)^2},
\qquad
\delta_\theta^2 V_{j,k}
=
\frac{V_{j,k+1}-2V_{j,k}+V_{j,k-1}}{(\Delta\theta)^2}.
\]
The homogeneous Neumann boundary condition in \(x\) is imposed by the ghost
values
\[
V_{-1,k}=V_{1,k},
\qquad
V_{N_x+1,k}=V_{N_x-1,k},
\]
while all indices in the \(\theta\)-direction are understood periodically.
For the phase variable, we also use
\[
\delta_\theta^-u_k
=
\frac{u_k-u_{k-1}}{\Delta\theta},
\qquad
\delta_\theta^+u_k
=
\frac{u_{k+1}-u_k}{\Delta\theta}.
\]

The semi-discrete WKB system involves two numerical ingredients that are not
contained in the standard second-order difference operators introduced above.
The first one is the numerical Hamiltonian for the phase equation. The
Hamilton--Jacobi term \(-|\partial_\theta u_\varepsilon|^2\) in \eqref{eq:WKB_reformulation} is approximated by
\[
\widehat H(\delta_\theta^-u_k^\varepsilon,\delta_\theta^+u_k^\varepsilon).
\]
We assume that \(\widehat H\) is locally Lipschitz and consistent with the
continuous Hamiltonian, namely
\[
\widehat H(p,p)=-p^2 .
\]
We also assume that it is monotone in the sense that it is nondecreasing with
respect to its first argument and nonincreasing with respect to its second
argument. When \(\widehat H\) is differentiable, this condition reads
\[
\partial_a\widehat H(a,b)\ge0,
\qquad
\partial_b\widehat H(a,b)\le0 .
\]
One typical example is the Godunov numerical Hamiltonian,
shown in Appendix~\ref{app:hamiltonian_flux}.

%The second ingredient is the conservative discretization of the transport
%coupling in the amplitude equation. We use \(\widehat{\mathcal F}^{\varepsilon}\) to approximate the flux
%\(F_\varepsilon= -W_\varepsilon\partial_\theta u_\varepsilon\) and denote 
%\[
%\mathcal D_\theta\widehat{\mathcal F}_{j,k}^\varepsilon
%\]
%to be a conservative and consistent discretization of \(\partial_\theta F_\varepsilon\) at
%\((x_j,\theta_k)\). 
%One common choice is the upwind manner in Appendix~\ref{app:hamiltonian_flux}.
%For the positivity argument below, we further assume
%the quasi-positivity property
%\begin{equation}\label{ass:H_flux}
%W_{j,k}=0,
%\qquad
%W_{j,\ell}\ge0\ \text{for all }\ell
%\quad\Longrightarrow\quad
%-\mathcal D_\theta\widehat{\mathcal F}(W,u)_{j,k}\ge0 .
%\end{equation}
%This property is satisfied by standard monotone finite volume fluxes, such as

The second ingredient is a conservative discretization of the transport
coupling in the amplitude equation.  Recall that the  flux in \eqref{eq:WKB_reformulation} is
$
F_\varepsilon=-W_\varepsilon\partial_\theta u_\varepsilon .
$
Let $\widehat{\mathcal F}^{\varepsilon}$ be a numerical flux approximating
$F_\varepsilon$. 
We denote by
$$
\mathcal D_\theta \widehat{\mathcal F}^{\varepsilon}_{j,k}
%=
%\mathcal D_\theta \widehat{\mathcal F}(W^\varepsilon,u^\varepsilon)_{j,k}
$$
a conservative and consistent approximation of
$\partial_\theta F_\varepsilon$ at the grid point $(x_j,\theta_k)$.  A typical choice is the upwind flux described in
Appendix~\ref{app:hamiltonian_flux}.

For the positivity argument below, we further assume that this discretization satisfies
the quasi-positivity property. 
Specifically,  for any fixed pair of indices $(j,k)$, if $W_{j,k}=0$ and $W_{j,k'}\ge0$ for all $k'$, then we have 
\begin{equation}\label{ass:H_flux}
-\mathcal D_\theta\widehat{\mathcal F}^{\varepsilon}_{j,k}\ge0 .
\end{equation}
%if $W_{j,k}=0$ and $W_{j,\ell}\ge0$ for all trait indices $\ell$, 
%then we have
%$-\mathcal D_\theta\widehat{\mathcal F}_{j,k}^\varepsilon\ge0$.
It is easy to check that the upwind flux described in Appendix~\ref{app:hamiltonian_flux} satisfies the property.

\paragraph{The semi-discrete WKB system.}
With these two ingredients specified, we now define the discrete density and
the semi-discrete WKB system. For notational convenience, set
\begin{equation}\label{def:E}
E_k^\varepsilon(t)
=
\exp\!\left(\frac{u_k^\varepsilon(t)}{\varepsilon}\right).
\end{equation}
The reconstructed nodal density is defined by
\begin{equation}\label{def:n_reconstructed}
n_{j,k}^\varepsilon(t)
=
W_{j,k}^\varepsilon(t)E_k^\varepsilon(t).
\end{equation}

As we shall see later, the evolution of \(u_\varepsilon\) and
\(W_\varepsilon\) is coupled through the spatial marginal density $\rho_\varepsilon$.  In the
semi-discrete scheme, this marginal is computed by the composite quadrature in
the trait direction.  More precisely, we define
\begin{equation}\label{eq:E_rho}
\rho_{j}^\varepsilon(t)
:=
\Delta\theta
\sum_{k=1}^{N_\theta}
n_{j,k}^\varepsilon(t)
=
\Delta\theta
\sum_{k=1}^{N_\theta}
W_{j,k}^\varepsilon(t)E_k^\varepsilon(t),
\qquad
j=0, 1,\cdots, N_x, 
\end{equation}
which approximates $\rho_{\varepsilon}(t, x_j)$ in \eqref{def:rho_P}. 
Given
\[
\boldsymbol\rho_\varepsilon(t)
=
\bigl(\rho_{0}^\varepsilon(t), \rho_{1}^\varepsilon(t),\cdots,\rho_{N_x}^\varepsilon(t)\bigr),
\]
the discrete effective Hamiltonian
\(\mathcal H_k(\boldsymbol\rho_\varepsilon(t))\), which approximates the principal eigenvalue $\mathcal H(\theta_k, \rho_\varepsilon(t))$ in \eqref{eq:eigen_N}, is determined by the discrete
principal eigenvalue problem
\begin{equation}\label{eq:eigen_discrete}
\left\{
\begin{aligned}
&-D_k\delta_x^2\mathcal N_j^{(k)}
=
\mathcal N_j^{(k)}
\bigl(K_j-\rho_{j}^\varepsilon(t)\bigr)
+
\mathcal N_j^{(k)}
\mathcal H_k\bigl(\boldsymbol\rho_\varepsilon(t)\bigr),
\qquad
j=0, 1,\cdots,N_x, \\
& \mathcal N_{-1}^{(k)} = \mathcal N_{1}^{(k)}, \quad \mathcal N_{N_x+1}^{(k)} = \mathcal N_{N_x-1}^{(k)}, 
\end{aligned}
\right.
\end{equation}
where \(\mathcal N^{(k)}\) is the corresponding positive eigenvector. 
%Its normalization does not affect \(\mathcal H_k\). For definiteness, one may impose
%\[
%\Delta x\sum_{j} \mathcal N_j^{(k)}=1 .
%\]

With the above notation, the semi-discrete WKB scheme reads
\begin{equation}\label{eq:semi_discrete_WKB_system}
\left\{
\begin{aligned}
&\frac{\mathrm d}{\mathrm dt}u_k^\varepsilon
+
\widehat H
\left(
\delta_\theta^-u_k^\varepsilon,
\delta_\theta^+u_k^\varepsilon
\right)
-
\varepsilon\delta_\theta^2u_k^\varepsilon
=
-\mathcal H_k\bigl(\boldsymbol\rho_\varepsilon(t)\bigr),
\qquad
k=1,\cdots,N_\theta,
\\[1mm]
&\varepsilon
\frac{\mathrm d}{\mathrm dt}W_{j,k}^\varepsilon
-
D_k\delta_x^2W_{j,k}^\varepsilon
-
\varepsilon^2\delta_\theta^2W_{j,k}^\varepsilon
+
2\varepsilon
\mathcal D_\theta
\widehat{\mathcal F}_{j,k}
\\
&\qquad =
W_{j,k}^\varepsilon
\bigl(
K_j-\rho_{j}^\varepsilon(t)
+\mathcal H_k(\boldsymbol\rho_\varepsilon(t))
-2\varepsilon\delta_\theta^2u_k^\varepsilon
\bigr),
\qquad
j=0, 1,\cdots,N_x,\quad k=1,\cdots,  N_\theta.
\end{aligned}
\right.
\end{equation}

\paragraph{The semi-discrete equation of $n$ and the remainder $R_{j,k}^\varepsilon$.} We next show the equation satisfied by the reconstructed density \(n_{j,k}^\varepsilon\). Since \(E_k^\varepsilon\) is independent of the
spatial index \(j\),
\[
\delta_x^2 n_{j,k}^\varepsilon
=
E_k^\varepsilon\delta_x^2W_{j,k}^\varepsilon .
\]
Moreover,
\begin{equation}\label{def:dt_n}
\varepsilon\frac{\mathrm d}{\mathrm dt}n_{j,k}^\varepsilon
=
E_k^\varepsilon
\left(
\varepsilon\frac{\mathrm d}{\mathrm dt}W_{j,k}^\varepsilon
+
W_{j,k}^\varepsilon
\frac{\mathrm d}{\mathrm dt}u_k^\varepsilon
\right).
\end{equation}
Multiplying the first equation in \eqref{eq:semi_discrete_WKB_system} by $E_k^\varepsilon W_{j,k}^\varepsilon$ and multiplying the second equation in  \eqref{eq:semi_discrete_WKB_system} by $ E_k^\varepsilon$, 
using \eqref{def:dt_n},
we can derive that \(n_{j,k}^\varepsilon\) satisfies
\begin{equation}\label{eq:semi_discrete_n}
\varepsilon\frac{\mathrm d}{\mathrm dt} n_{j,k}^\varepsilon
-
D_k\delta_x^2 n_{j,k}^\varepsilon
=
\bigl(K_j-\rho_{j}^\varepsilon\bigr)n_{j,k}^\varepsilon
+
R_{j,k}^\varepsilon ,
\end{equation}
where the remainder has the form 
\begin{align} 
R_{j,k}^\varepsilon
&=
\varepsilon^2
E_k^\varepsilon
\delta_\theta^2 W_{j,k}^\varepsilon
-
\varepsilon
W_{j,k}^\varepsilon
E_k^\varepsilon
\delta_\theta^2 u_k^\varepsilon
-
W_{j,k}^\varepsilon
E_k^\varepsilon
\widehat H
\left(
\delta_\theta^-u_k^\varepsilon,
\delta_\theta^+u_k^\varepsilon
\right)
-
2\varepsilon
E_k^\varepsilon
\mathcal D_\theta
\widehat{\mathcal F}^\varepsilon_{j,k}.
\label{eq:R_exact}
\end{align}
The equation \eqref{eq:semi_discrete_n} is a semi-discrete analogue of the
original density equation \eqref{eq:n_intro}, with the remainder \(R_{j,k}^\varepsilon\) \eqref{eq:R_exact} which approximates
\(\varepsilon^2\partial_{\theta\theta}n^\varepsilon\).

%===================================================================================================
%   Subsection: basic structural property
%===================================================================================================

\subsection{Basic structural property}

The semi-discrete scheme \eqref{eq:semi_discrete_WKB_system} can be shown to be positive-preserving 
as long as the initial data is nonnegative. 
The detailed result is summarized as the following lemma. 

\begin{lemma}[Positivity of the amplitude]
\label{lem:positivity_W}
Assume that \(W_{j,k}^\varepsilon(0)\ge0\) for all \(j,k\). Then the
semi-discrete amplitude equation in \eqref{eq:semi_discrete_WKB_system}
preserves nonnegativity, namely
\[
    W_{j,k}^\varepsilon(t)\ge0
    \qquad
    \forall\,j,k,\ t\ge0,
\]
as long as the solution exists. Consequently,
\[
    n_{j,k}^\varepsilon(t)\ge0
    \qquad
    \forall\,j,k,\ t\ge0 .
\]
\end{lemma}

\begin{proof}
%It is enough to verify the quasi-positivity of the right-hand side of the
%finite-dimensional ODE system for \(W^\varepsilon\). 
From the amplitude
equation in \eqref{eq:semi_discrete_WKB_system}, we can write
\[
\varepsilon \frac{\mathrm d}{\mathrm dt}W_{j,k}^\varepsilon
=
D_k\delta_x^2W_{j,k}^\varepsilon
+
\varepsilon^2\delta_\theta^2W_{j,k}^\varepsilon
-
2\varepsilon
\mathcal D_\theta\widehat{\mathcal F}^\varepsilon_{j,k}
+
W_{j,k}^\varepsilon B_{j,k}(t),
\]
where
\[
B_{j,k}(t)
=
K_j-\rho_{j}^\varepsilon(t)
+\mathcal H_k(\boldsymbol\rho_\varepsilon(t))
-2\varepsilon\delta_\theta^2u_k^\varepsilon(t).
\]
Let \(W^\varepsilon\ge0\) and suppose that
\(W_{j,k}^\varepsilon=0\) for some \((j,k)\) at time $t$. Then, for an interior point,
\[
\delta_x^2W_{j,k}^\varepsilon
=
\frac{W_{j+1,k}^\varepsilon+W_{j-1,k}^\varepsilon}{(\Delta x)^2}
\ge0, 
\quad 
\delta_\theta^2W_{j,k}^\varepsilon
=
\frac{W_{j,k+1}^\varepsilon+W_{j,k-1}^\varepsilon}{(\Delta\theta)^2}
\ge0 .
\]
The same conclusion holds at the boundary because of the Neumann ghost
values in $x$-direction and the periodicity in \(\theta\)-direction. 
Combining all the inequalities, using the quasi-positivity assumption \eqref{ass:H_flux} and the fact that 
\(
    W_{j,k}^\varepsilon B_{j,k}(t)=0 
\)
since we assumed $W_{j,k}^\varepsilon=0$, 
we finally have that, at this specific $(j,k)$, the following inequality holds, 
\[
\left.
\frac{\mathrm d}{\mathrm dt}W_{j,k}^\varepsilon
\right|_{W_{j,k}^\varepsilon=0}
\ge0 .
\]
Thus the vector field points inward on the boundary of the nonnegative cone.
The standard invariance criterion for finite-dimensional ODE systems implies
the claim.
\end{proof}

%\begin{remark}[On the time-dependent stability]
%For each fixed \(\varepsilon>0\), the semi-discrete system is a
%finite-dimensional ODE system under the locally Lipschitz assumptions on the
%discrete operators. Uniform-in-\(\varepsilon\) continuation estimates for the
%time-dependent problem require a separate bootstrap argument involving the
%density, phase, amplitude, and trait-direction remainder. This issue is not
%pursued in the present work.
%\end{remark}

%===================================================================================================
%   Subsection: discrete stationary asymptotic preserving structure
%===================================================================================================

\subsection{Conditional stationary asymptotic-preserving structure}
\label{subsec:stationary_discrete_asymptotic preserving}

We now turn to stationary semi-discrete states. The goal is to identify which
parts of the continuous rare-mutation structure are retained on a fixed grid.
%For the generic split WKB scheme, the only nonstandard point is the
%trait-direction remainder in the reconstructed density equation. We therefore
%state the stationary result under an explicit one-sided stability condition.
%This keeps the asymptotic preserving statement conditional and avoids claiming that the WKB
%remainder is automatically controlled for every trait discretization.
By Lemma~\ref{lem:positivity_W}, stationary states obtained from nonnegative
amplitudes satisfy
\[
    W_{j,k}^\varepsilon\ge0,
    \qquad
    n_{j,k}^\varepsilon
    =W_{j,k}^\varepsilon E_k^\varepsilon\ge0 .
\]
All estimates below are for such nonnegative reconstructed densities.

\paragraph{Weighted density balance.}
Define
\begin{equation}\label{def:eta}
    \eta_j^\varepsilon
    =
    \Delta\theta
    \sum_{k=1}^{N_\theta}
    \frac{n_{j,k}^\varepsilon}{D_k}.
\end{equation}
Recalling \eqref{eq:E_rho}, multiplying \eqref{eq:semi_discrete_n} by \(1/D_k\) and summing over the
trait variable gives the stationary weighted balance equation
\begin{equation}\label{eq:weighted_summed_density_stationary}
    -\delta_x^2\rho_j^\varepsilon
    =
    \eta_j^\varepsilon
    \bigl(K_j-\rho_j^\varepsilon\bigr)
    +
    \mathcal R_j^{\varepsilon,D},
    \qquad j=0, 1,\cdots,N_x,
\end{equation}
where
\begin{equation}\label{def:R_D}
    \mathcal R_j^{\varepsilon,D}
    =
    \Delta\theta
    \sum_{k=1}^{N_\theta}
    \frac{R_{j,k}^\varepsilon}{D_k}.
\end{equation}
The term \(\mathcal R_j^{\varepsilon,D}\) is the discrete analogue of the
continuous weighted mutation contribution
\begin{equation}\label{eq:cts_R_expression}
    \varepsilon^2
    \int_{\mathbb T}
    \frac1{D(\theta)}
    \partial_{\theta\theta}n^\varepsilon
    \,d\theta
    =
    \varepsilon^2
    \int_{\mathbb T}
    n^\varepsilon
    \partial_{\theta\theta}\left(\frac1{D(\theta)}\right)
    \,d\theta ,
\end{equation}
where periodicity in the trait variable has been used.  At the continuous
level this term is bounded above by \(C\varepsilon^2\rho^\varepsilon\).  For
the split WKB discretization, however, the discrete operators act on the
reconstructed density
$
    n_{j,k}^\varepsilon
$
and an exact discrete chain rule is not available.  
Therefore, we present the following one-sided control assumption instead.

\begin{assumption}[One-sided weighted remainder stability]
\label{ass:remainder_stability}
For fixed spatial and trait grids, there exist constants \(C_{R,h}>0\) and \(r_h\ge0\),
possibly depending on the fixed mesh but independent of \(\varepsilon\), such that every
stationary state under consideration satisfies
\begin{equation}\label{eq:remainder_CR_bound}
    \bigl(\mathcal R_j^{\varepsilon,D}\bigr)_+
    \le
    C_{R,h}\rho_j^\varepsilon+r_h,
    \quad j=0,1,\cdots,N_x.
\end{equation}
\end{assumption}

This condition controls only the positive part of the weighted defect.  A
negative contribution is favorable in the maximum-principle estimate. 
 The assumption should be understood as a structural stability condition on the trait
discretization. 
For a density-level conservative discretization, this type of estimate follows
from periodic discrete summation by parts in the trait variable. Appendix B
records this special case.
For the WKB scheme used in the paper,  the behavior of this estimate is examined numerically in
Section~\ref{sec:numerics}.  
As we will see later, in the numerical tests, the weighted defect is
typically nonpositive or has a small positive part, which supports the practical
validity of the assumption.

\begin{proposition}[Conditional density bound for stationary semi-discrete states]
\label{prop:uniform_bound_stationary_mass}
Let \((u^\varepsilon,W^\varepsilon)\) be a stationary semi-discrete state with
nonnegative reconstructed density.  Assume that
Assumption~\ref{ass:remainder_stability} holds.  Then there exists a constant
\(C_\rho>0\), independent of \(\varepsilon\), such that
\[
    0\le \rho_j^\varepsilon\le C_\rho,
    \qquad j=0, 1,\cdots, N_x .
\]
\end{proposition}

\begin{proof}
By using  \(K_j\le K_{\max}\) and Assumption~\ref{ass:remainder_stability}, it is easy to get from \eqref{eq:weighted_summed_density_stationary} that 
\[
    -\delta_x^2\rho_j^\varepsilon
    +
    \eta_j^\varepsilon\rho_j^\varepsilon
    =
    K_j  \eta_j^\varepsilon
    +
    \mathcal R_j^{\varepsilon,D}
    \le
    K_{\max}\eta_j^\varepsilon
    +
    C_{R,h}\rho_j^\varepsilon
    +
    r_h .
\] 
Let \(j_*\) be a point where
\[
    \rho_{j_*}^\varepsilon = \max_j \rho_j^\varepsilon.
\]
Since \(\delta_x^2\rho_{j_*}^\varepsilon\le0\), we have
\begin{equation}\label{proof:rho_inequality}
    \eta_{j_*}^\varepsilon \rho_{j_*}^\varepsilon
    \le
    K_{\max}\eta_{j_*}^\varepsilon
    +
    C_{R,h}\rho_{j_*}^\varepsilon
    +
    r_h .
\end{equation}
On the other hand, recalling \eqref{eq:basic_assumptions} and \eqref{def:eta}, it is easy to see that, for all $j$, 
\begin{equation}\label{proof:rho_eta}
    \frac1{D_{\max}}\rho_j^\varepsilon
    \le
    \eta_j^\varepsilon
    \le
    \frac1{D_{\min}}\rho_j^\varepsilon .
\end{equation}
Combining \eqref{proof:rho_inequality} and \eqref{proof:rho_eta} gives
\[
    \frac{(\rho_{j_*}^\varepsilon)^2}{D_{\max}}
    \le
    \left(
        \frac{K_{\max}}{D_{\min}}+C_{R,h}
    \right)\rho_{j_*}^\varepsilon
    +
    r_h .
\]
The right-hand side is at most linear in \(\rho_{j_*}^\varepsilon\), while the left-hand
side is quadratic.  Hence \(\rho_{j_*}^\varepsilon\) is bounded by a constant independent
of \(\varepsilon\).  This proves the claim.
\end{proof}

\paragraph{Boundedness of the effective Hamiltonian.}
The density bound controls the coefficients in the discrete eigenvalue problem
and hence the stationary phase equation.

%\begin{proposition}[Uniform bound for the discrete effective Hamiltonian]
%\label{prop:stationary_H_bound}
%Assume that
%\[
%    0\le \rho_j^\varepsilon\le C_\rho,
%    \qquad j=0, 1,\cdots, N_x,
%\]
%where \(C_\rho\) is independent of \(\varepsilon\).  Then the discrete
%effective Hamiltonian defined by \eqref{eq:eigen_discrete} satisfies
%\[
%    |\mathcal H_k(\boldsymbol\rho^\varepsilon)|
%    \le C_{\mathcal H},
%    \qquad k=1,\cdots, N_\theta,
%\]
%where \(C_{\mathcal H}\) is independent of \(\varepsilon\).
%\end{proposition}
%
%\begin{proof}
%The discrete eigenvalue problem is associated with the symmetric operator
%\[
%    -D_k\delta_x^2-(K_j-\rho_j^\varepsilon).
%\]
%Its principal eigenvalue has the Rayleigh representation
%\[
%    \mathcal H_k(\boldsymbol\rho^\varepsilon)
%    =
%    \min_{\varphi\ne0}
%    \frac{
%        D_k\Delta x\sum_{j=0}^{N_x-1} |\delta_x^+\varphi_j|^2
%        -
%        \Delta x\sum_{j=0}^{N_x}  \omega_j
%        (K_j-\rho_j^\varepsilon)\varphi_j^2
%    }{
%        \Delta x\sum_{j=0}^{N_x} \omega_j \varphi_j^2
%    } ,
%\]
%where $\omega_j =  \frac{1}{2}$ for $j=0, N_x$ and $\omega_j = 1$ for $1\le j \le N_x-1$. 
%On the one hand, since \(0\le K_j\le K_{\max}\) and
%\(0\le\rho_j^\varepsilon\le C_\rho\), we have
%\[
%    \mathcal H_k(\boldsymbol\rho^\varepsilon)\ge -K_{\max}.
%\]
%On the other hand, testing the quotient with a constant vector gives
%\[
%    \mathcal H_k(\boldsymbol\rho^\varepsilon)\le C_\rho .
%\]
%Thus
%\[
%    |\mathcal H_k(\boldsymbol\rho^\varepsilon)|
%    \le
%    \max\{K_{\max},C_\rho\}
%    =:C_{\mathcal H},
%\]
%which proves the claim.
%\end{proof}

\begin{proposition}[Uniform bound for the discrete effective Hamiltonian]
\label{prop:stationary_H_bound}
Assume that
\[
    0\le \rho_j^\varepsilon\le C_\rho,
    \qquad j=0,1,\cdots,N_x,
\]
where \(C_\rho\) is independent of \(\varepsilon\). Then the discrete
effective Hamiltonian defined by \eqref{eq:eigen_discrete} satisfies
\[
    |\mathcal H_k(\boldsymbol\rho^\varepsilon)|
    \le C_{\mathcal H},
    \qquad k=1,\cdots,N_\theta,
\]
where \(C_{\mathcal H}\) is independent of \(\varepsilon\).
\end{proposition}

\begin{proof}
We use the trapezoidal weighted inner product
\[
    \langle \varphi,\psi\rangle_h
    =
    \Delta x\sum_{j=0}^{N_x}\omega_j \varphi_j\psi_j,
    \qquad
    \omega_0=\omega_{N_x}=\frac12,\qquad
    \omega_j=1,\quad 1\le j\le N_x-1 .
\]
With the Neumann ghost values
\[
    \varphi_{-1}=\varphi_1,\qquad
    \varphi_{N_x+1}=\varphi_{N_x-1},
\]
the operator \(-\delta_x^2\) is self-adjoint and nonnegative with respect to
this weighted inner product. More precisely,
\[
    \langle -\delta_x^2\varphi,\varphi\rangle_h
    =
    \Delta x\sum_{j=0}^{N_x-1}
    |\delta_x^+\varphi_j|^2,
    \qquad
    \delta_x^+\varphi_j
    =
    \frac{\varphi_{j+1}-\varphi_j}{\Delta x}.
\]
Hence the discrete eigenvalue problem is associated with the self-adjoint
operator
\[
    -D_k\delta_x^2-\operatorname{diag}(K_j-\rho_j^\varepsilon)
\]
in this weighted inner product. Its principal eigenvalue has the Rayleigh
representation
\[
    \mathcal H_k(\boldsymbol\rho^\varepsilon)
    =
    \min_{\varphi\ne0}
    \frac{
        D_k\Delta x\sum_{j=0}^{N_x-1} |\delta_x^+\varphi_j|^2
        -
        \Delta x\sum_{j=0}^{N_x} \omega_j
        (K_j-\rho_j^\varepsilon)\varphi_j^2
    }{
        \Delta x\sum_{j=0}^{N_x} \omega_j \varphi_j^2
    } .
\]
On the one hand, since \(0\le K_j\le K_{\max}\) and
\(\rho_j^\varepsilon\ge0\), the nonnegative gradient term gives
\[
    \mathcal H_k(\boldsymbol\rho^\varepsilon)
    \ge -K_{\max}.
\]
On the other hand, testing the quotient with the constant vector
\(\varphi_j\equiv1\) gives
\[
    \mathcal H_k(\boldsymbol\rho^\varepsilon)
    \le
    \frac{
        -\Delta x\sum_{j=0}^{N_x}\omega_j
        (K_j-\rho_j^\varepsilon)
    }{
        \Delta x\sum_{j=0}^{N_x}\omega_j
    }
    \le C_\rho ,
\]
because \(K_j\ge0\) and \(\rho_j^\varepsilon\le C_\rho\). Therefore
\[
    |\mathcal H_k(\boldsymbol\rho^\varepsilon)|
    \le
    \max\{K_{\max},C_\rho\}
    =:C_{\mathcal H},
\]
which proves the claim.
\end{proof}

For the stationary phase equation, the Godunov numerical Hamiltonian further gives a
fixed-grid bound for the phase.  To be more specific, suppose that
\begin{equation}\label{def:Godunov_eq}
    \widehat H_G
    \left(
        \delta_\theta^-u_k^\varepsilon,
        \delta_\theta^+u_k^\varepsilon
    \right)
    -
    \varepsilon\delta_\theta^2u_k^\varepsilon
    =
    -\mathcal H_k(\boldsymbol\rho^\varepsilon),
\end{equation}
where $\widehat H_G$ is defined in \eqref{eq:app_Godunov_H}.  
%\begin{equation}\label{def:Godunov}
%    \widehat H_G(p^-,p^+)
%    =
%    -\bigl((p^-)_{-}\bigr)^2
%    -
%    \bigl((p^+)_{+}\bigr)^2 .
%\end{equation}
%Set \(q_k^\varepsilon=D_\theta^+u_k^\varepsilon\).  Then
%\(D_\theta^-u_k^\varepsilon=q_{k-1}^\varepsilon\), and periodicity gives
%\[
%    -\widehat H_G
%    \left(
%        D_\theta^-u_k^\varepsilon,
%        D_\theta^+u_k^\varepsilon
%    \right)
%    =
%    \bigl((q_{k-1}^\varepsilon)_{-}\bigr)^2
%    +
%    \bigl((q_k^\varepsilon)_{+}\bigr)^2 .
%\]
Summing the stationary phase equation over the periodic trait grid eliminates
\(\delta_\theta^2u^\varepsilon\).  Therefore
\[
    \Delta\theta
    \sum_{k=1}^{N_\theta} |\delta_\theta^+u_k^\varepsilon|^2
    =
    -\Delta\theta
    \sum_{k=1}^{N_\theta}
    \widehat H_G
    \left(
        \delta_\theta^-u_k^\varepsilon,
        \delta_\theta^+u_k^\varepsilon
    \right)
    =
    \Delta\theta
    \sum_{k=1}^{N_\theta}
    \mathcal H_k(\boldsymbol\rho^\varepsilon)
    \le C_{\mathcal H}.
\]
Thus
\begin{equation}\label{proof:bound_Du}
    \max_k |\delta_\theta^+u_k^\varepsilon|
    \le
    \frac{\sqrt{C_{\mathcal H}}}{\sqrt{\Delta\theta}} .
\end{equation}

\paragraph{Formal asymptotic-preserving trait-selection limit.}

We now pass formally to the fixed-grid rare-mutation limit.  The density bound
gives compactness of the discrete spatial marginal, and the Godunov phase
estimate gives compactness of the normalized phase.  Passing to the limit in
the stationary phase equation yields a discrete constrained Hamilton--Jacobi
structure.

\begin{proposition}[Formal asymptotic-preserving structure with Godunov Hamiltonian]
\label{prop:formal_asymptotic preserving_trait_selection}
Assume that the stationary semi-discrete states satisfy
Assumption~\ref{ass:remainder_stability}, and normalize the phase by
\[
    \max_k u_k^\varepsilon=0 .
\]
Assume also that the stationary phase equation is discretized with
\(\widehat H_G\) defined in \eqref{eq:app_Godunov_H}.  Then, up to a subsequence on
the fixed spatial and trait grids,
\[
    \boldsymbol\rho^\varepsilon\to \boldsymbol\rho^0,
    \qquad
    \boldsymbol u^\varepsilon\to \boldsymbol u^0 .
\]
Moreover, the limiting phase satisfies the discrete constrained
Hamilton--Jacobi system
\[
    \widehat H_G
    \left(
        \delta_\theta^-u_k^0,
        \delta_\theta^+u_k^0
    \right)
    =
    -\mathcal H_k(\boldsymbol\rho^0),
    \qquad
    \max_k u_k^0=0 .
\]
If further \(\mathcal H_k(\boldsymbol\rho^0)\) has a unique zero at
\(k_m\) on the fixed trait grid, then
\[
    k^*  = k_m,
\]
where $ k^* = \arg\max_k u_k^0$, meaning that the selected grid trait is exactly \(\theta_{k_m}\).
\end{proposition}

\begin{proof}
Proposition~\ref{prop:uniform_bound_stationary_mass} gives a uniform bound on
\(\boldsymbol\rho^\varepsilon\).  Since the spatial grid is fixed, a subsequence
converges to some \(\boldsymbol\rho^0\).  The preceding Godunov phase estimate
and the normalization give a uniform bound on \(\boldsymbol u^\varepsilon\) on
the fixed trait grid.  Hence, up to a further subsequence,
\[
    \boldsymbol u^\varepsilon\to \boldsymbol u^0 .
\]
Passing to the limit in the stationary phase equation gives the limiting
discrete constrained Hamilton--Jacobi system, because
\(\varepsilon\delta_\theta^2u_k^\varepsilon\to0\) on the fixed trait grid and
\[
    \mathcal H_k(\boldsymbol\rho^\varepsilon)
    \to
    \mathcal H_k(\boldsymbol\rho^0).
\]
The normalization also passes to the limit.

It remains to identify the selected grid trait.  Let \(k_*\) be a maximizer of
\(u^0\).  Then
\[
    \delta_\theta^-u_{k_*}^0\ge0,
    \qquad
    \delta_\theta^+u_{k_*}^0\le0 .
\]
By the definition of the Godunov Hamiltonian,
\[
    \widehat H_G
    \left(
        \delta_\theta^-u_{k_*}^0,
        \delta_\theta^+u_{k_*}^0
    \right)
    =
    0 .
\]
The limiting phase equation therefore gives
\[
    \mathcal H_{k_*}(\boldsymbol\rho^0)=0 .
\]
If \(\mathcal H_k(\boldsymbol\rho^0)\) has a unique zero at \(k_m\), then every
maximizer \(k_*\) of \(u^0\) must satisfy \(k_*=k_m\). 
\end{proof}

The fixed-grid asymptotic-preserving structure obtained above is the discrete analogue of the
continuous rare-mutation structure summarized in
Theorem~\ref{prop:continuous_structure}.  The discrete density estimate in Proposition \ref{prop:uniform_bound_stationary_mass}
plays the role of the continuous bound on the spatial marginal.  The discrete
eigenvalue problem \eqref{eq:eigen_discrete} defines the numerical effective Hamiltonian, corresponding
to the continuous principal eigenvalue problem.  
By Proposition \ref{prop:formal_asymptotic preserving_trait_selection}, the limiting discrete
Hamilton--Jacobi system identifies the selected grid trait through the maximizer
of the phase, in the same way that the continuous constrained
Hamilton--Jacobi equation identifies the selected trait in the rare-mutation
limit.  Thus the semi-discrete WKB scheme preserves the essential selection
structure on fixed spatial and trait grids.

The preceding argument is a stationary fixed-grid asymptotic-preserving structure.  In the
computations, the time-dependent WKB system is used as a relaxation dynamics
toward such stationary states.  Thus the semi-discrete analysis identifies the
limiting stationary structure, while the numerical algorithm evolves the WKB
variables toward this discrete eigenvalue and constrained Hamilton--Jacobi
structure.

%===================================================================================================
%	Numerical experiments
%===================================================================================================

\section{Numerical experiments}\label{sec:numerics}

This section presents the fully discrete implementation and numerical
experiments for the WKB formulation.  We focus on its behavior in the
small-mutation regime and compare it with direct density discretizations to
demonstrate the advantage of the WKB formulation.

\subsection{Fully discrete implementation and dual trait-grid implementation}
\label{subsec:dual_grid}

We first describe the time discretization used in the numerical experiments.
For comparison, the original density equation is discretized by the linear
implicit update
\begin{equation}\label{eq:direct_density_fully}
    \varepsilon
    \frac{n_{j,k}^{l+1}-n_{j,k}^l}{\Delta t}
    -
    D_k\delta_x^2n_{j,k}^{l+1}
    -
    \varepsilon^2\delta_\theta^2n_{j,k}^{l+1}
    =
    n_{j,k}^{l+1}(K_j-\rho_j^l),
    \qquad
    \rho_j^l=\Delta\theta\sum_{k=1}^{N_\theta} n_{j,k}^l .
\end{equation}
For the WKB formulation, the phase variable $u$ is advanced by
\begin{equation}\label{eq:wkb_fully_u}
    \frac{u_k^{l+1}-u_k^l}{\Delta t}
    +
    \widehat H(\delta_\theta^-u_k^l,\delta_\theta^+u_k^l)
    -
    \varepsilon\delta_\theta^2u_k^{l+1}
    =
    -H_k^l, 
\end{equation}
where $H_k^l$ is determined by the discrete
principal eigenvalue problem \eqref{eq:eigen_discrete} 
with $ \rho_j^l$ reconstructed via \eqref{eq:E_rho}. 
With \(u^{l+1}\) fixed, the amplitude $W$ is advanced by the linear implicit
update
\begin{equation}\label{eq:wkb_fully_W}
\begin{aligned}
    &\varepsilon
    \frac{W_{j,k}^{l+1}-W_{j,k}^l}{\Delta t}
    -
    D_k\delta_x^2W_{j,k}^{l+1}
    -
    \varepsilon^2\delta_\theta^2W_{j,k}^{l+1}
    +
    2\varepsilon D_\theta\widehat F(W^{l+1},u^{l+1})_{j,k}
    \\
    &\qquad
    =
    W_{j,k}^{l+1}
    \left(
        K_j-\rho_j^l+H_k^l
        -2\varepsilon\delta_\theta^2u_k^{l+1}
    \right).
\end{aligned}
\end{equation}
Since \(u^{l+1}\) is already known, the amplitude update is also linear.

\paragraph{Large time-step relaxation.}
The time-dependent problem is used in this work as a relaxation dynamics
toward the stationary selection state. In the small-mutation regime, the
amplitude equation contains a fast spatial relaxation. Indeed, for frozen
$\rho^l$ and $H_k^l$, the leading $x$-dependent part of the amplitude equation
is governed by the shifted operator
$$
A_k(\rho^l)-H_k^l I,
\qquad
A_k(\rho^l)=-D_k\delta_x^2-\operatorname{diag}(K-\rho^l),
$$
whose null space is generated by the positive principal eigenvector associated
with the discrete effective Hamiltonian. If
$\mu_{\ell,k}^l>0$ denotes a nonzero eigenvalue of
$A_k(\rho^l)-H_k^l I$, where $\ell$ indexes the non-principal spatial modes,
then the implicit treatment in \eqref{eq:wkb_fully_W} damps this mode by a
factor of the form
$$
\left(1+\frac{\Delta t}{\varepsilon}\mu_{\ell,k}^l\right)^{-1}.
$$
Therefore, when $\varepsilon$ is very small, the WKB relaxation can use time
steps much larger than $\varepsilon$ to drive the amplitude rapidly toward the
current principal-eigenvector profile. This is a practical advantage of the WKB
relaxation solver.

\paragraph{Mass correction.}
In the time-dependent computations, we also use a scalar mass correction
during the relaxation process. Integrating the density equation over
\(x\) and \(\theta\), and using the homogeneous Neumann boundary condition
in \(x\) and the periodic boundary condition in \(\theta\), gives
\[
\varepsilon \frac{d}{dt} M_\varepsilon(t)
=
\int_\Omega
\rho_\varepsilon(t,x)
\bigl(K(x)-\rho_\varepsilon(t,x)\bigr)\,dx,
\qquad
M_\varepsilon(t)
=
\int_\Omega\int_{\mathbb T}
n_\varepsilon(t, x,\theta)\,d\theta dx .
\]
After each time step, the provisional density is rescaled by a positive scalar
so that the discrete total mass is consistent with the above mass evolution
equation. In the computations below, this scalar factor is obtained using a
trapezoidal rule in time. For the direct density solver, this amounts to
rescaling \(n^{l+1}\). For the WKB solver, we rescale the amplitude
\(W^{l+1}\), while the phase \(u^{l+1}\) is left unchanged. This correction is
used only to enforce this global mass relation during the relaxation and does not
change the location of the selected trait.

\paragraph{Dual trait-grid implementation.} To improve computational efficiency, we introduce the dual trait-grid
implementation.  The phase variable \(u\) and the amplitude variable \(W\) are
represented on different trait grids.  Let \(\Theta_u\) be the fine phase grid
and \(\Theta_W\) the coarser amplitude grid.  When
\(\Theta_u=\Theta_W\), the method reduces to the standard single-grid WKB
implementation.  When \(\Theta_W\) is coarser, the amplitude \(W\) is stored and
evolved on fewer trait nodes, while the phase \(u\) still resolves the selected
trait on the finer one-dimensional grid.

The coupling between the two trait grids is handled by periodic interpolation
in the trait variable.  Since \(u\) depends only on the trait variable,
refining \(\Theta_u\) is relatively inexpensive.  By contrast, \(W\) depends on
both space and trait.  In \(d\) spatial dimensions, the number of unknowns
scales as
\[
    O\!\left(N_x^d N_\theta^{(W)}+N_\theta^{(u)}\right),
\]
up to the cost of the spatial solvers.  When $N_\theta^{(W)} \ll N_\theta^{(u)}$, the dual-grid strategy reduces the
dominant memory and CPU cost while retaining fine trait resolution for the
selection process.

\subsection{Numerical experiments in 1D}
\label{subsec:numeric_1D}
Unless otherwise stated, the one-dimensional tests use
\[
    \Omega=[0,1],
    \qquad
    \mathbb T=(0,1],
\]
with homogeneous Neumann boundary conditions in the spatial variable and
periodic boundary conditions in the trait variable.  
The carrying capacity and the dispersal rate are chosen as 
\begin{equation}\label{eq:num_K_and_D}
    K(x)=K_0-K_1\cos(2\pi x),
    \qquad
    D(\theta) = D_{\min} + D_1\bigl(1-\cos(2\pi(\theta-\theta_m))\bigr),
\end{equation}
respectively, with $  K_0=1, \, K_1=0.5, \, D_{\min}=0.2,   \, D_1=0.4, \,   \theta_m=0.35$. 
Thus \(D\) has a unique minimum at \(\theta_m\), which is the expected selected
trait in the rare-mutation limit. 
The default initial data are chosen as 
\begin{equation}\label{eq:num_initial_W_new}
    u_k^0=0,
    \qquad
    W_{j,k}^0=n_{j,k}^0=1 .
\end{equation}
Other initial data are specified in the individual tests when needed.

% ========================= Test 1 ============================
\subsubsection{Numerical verification of Assumption~\ref{ass:remainder_stability}}
\label{subsec:test_assumptions_new}

We first examine the one-sided weighted remainder condition used in the
stationary asymptotic-preserving argument.  
For each computation, after the solution has reached a numerical steady state,
we record the pairs
\[
    \left(\rho_j^\varepsilon,\mathcal R_j^{\varepsilon,D}\right),
    \qquad j=0, 1,\cdots, N_x, 
\]
where \(\mathcal R_j^{\varepsilon,D}\) is the weighted remainder \eqref{def:R_D}.

%Figure~\ref{fig:remainder_assumption} shows representative scatter plots of
%\(\mathcal R_j^{\varepsilon,D}\) against \(\rho_j^\varepsilon\) for several
%values of \(\varepsilon\) and for differenxt dispersal functions \(D(\theta)\).  
%In the displayed tests, the recorded weighted remainders
%\(\mathcal R_j^{\varepsilon,D}\) are negative. 
%This is consistent with \eqref{eq:cts_R_expression}, since for small
%$\varepsilon$ the density is concentrated near the selected trait, where $D(\theta)$ is
%minimal and therefore $1/D(\theta)$ has negative curvature. 
%Moreover, their absolute values
%remain small relative to \(\rho_j^\varepsilon\), so any possible positive
%numerical residual at a comparable scale would be negligible for the
%one-sided estimate.  
%This is also expected from \eqref{eq:cts_R_expression}, since the factor $\varepsilon^2$ is in front of the term.  
%This provides supporting numerical evidence for
%Assumption~\ref{ass:remainder_stability} and for the stationary
%asymptotic-preserving argument of the split WKB scheme.

Figure~\ref{fig:remainder_assumption} shows representative scatter plots of
$\mathcal R_j^{\varepsilon,D}$ against $\rho_j^\varepsilon$ for several
values of $\varepsilon$ and for different dispersal functions $D(\theta)$.
In the displayed tests, the recorded weighted remainders
$\mathcal R_j^{\varepsilon,D}$ are negative.  This behavior is consistent with
the continuous expression \eqref{eq:cts_R_expression}.  Indeed, for small
$\varepsilon$, the density is concentrated near the selected trait, where
$D(\theta)$ attains its minimum and, for the choices of $D$ considered here,
$1/D(\theta)$ has negative curvature.  Hence the dominant contribution in
\eqref{eq:cts_R_expression} is expected to be nonpositive.  Moreover, the
magnitudes of the recorded remainders remain small relative to
$\rho_j^\varepsilon$, which is also consistent with the prefactor
$\varepsilon^2$ in \eqref{eq:cts_R_expression}.  This provides supporting
numerical evidence for Assumption~\ref{ass:remainder_stability} and for the
stationary asymptotic-preserving argument of the WKB scheme.

\begin{figure}[htp]
\centering
\includegraphics[width=0.49\textwidth]{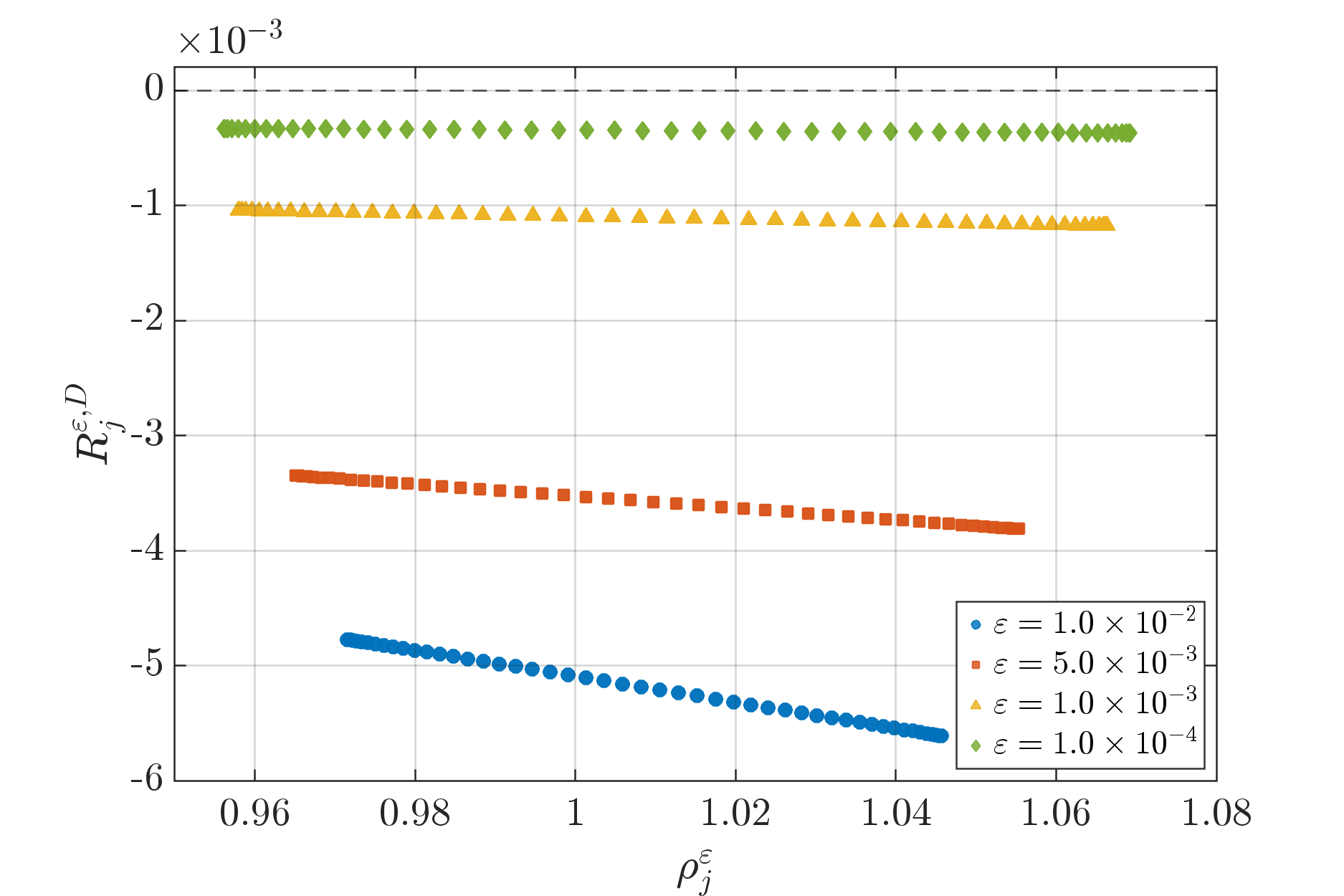}
\includegraphics[width=0.49\textwidth]{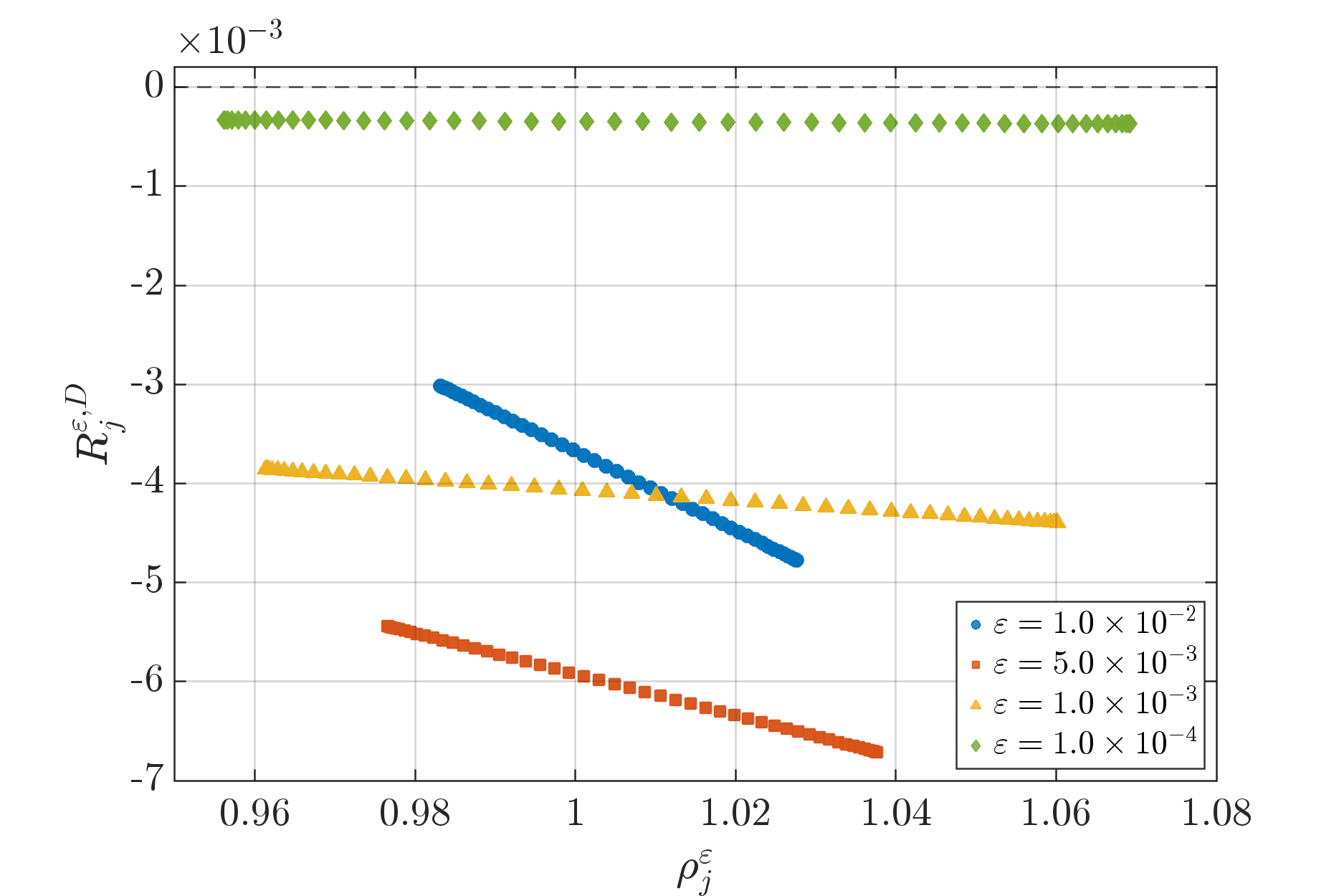}
\caption{Numerical verification of  Assumption~\ref{ass:remainder_stability}.
Left: \(D(\theta)=0.2+0.4\bigl(1-\cos(2\pi(\theta-\theta_m))\bigr)\).
Right:
\(D(\theta)=0.2+0.2\bigl(1-\cos(2\pi(\theta-\theta_m))\bigr)
+0.4\bigl(1-\cos(6\pi(\theta-\theta_m))\bigr)\).
The dashed horizontal line indicates zero.}
\label{fig:remainder_assumption}
\end{figure}

% ========================= Test 2 ============================
\subsubsection{Consistency between the direct density solver and the WKB formulation}
\label{subsec:test_fine_grid_n_new}

We next perform a same-grid consistency check at finite \(\varepsilon\).  The
direct density solver and the WKB solver are run on the same spatial and trait
grids.  The WKB density is reconstructed in the same way as in \eqref{def:n_reconstructed}, 
and the corresponding spatial marginal $\rho_\varepsilon(t,x)$ and trait marginal $P_{\varepsilon}(t, \theta)$, both defined in \eqref{def:rho_P}, are compared with
those obtained from the direct density computation.

Figure~\ref{fig:rho_pk_same_grid} shows the comparison at
 \(t=10\) with \(\varepsilon=10^{-3}\), \(N_x=100\), \(N_\theta=1024\) and $\Delta t = 10^{-3}$.  The two
solutions are almost indistinguishable for both the spatial marginal
\(\rho_\varepsilon(t=10, x)\) and the trait marginal \(P_\varepsilon(t=10, \theta)\).  This indicates that, when
the trait grid is sufficiently fine, the WKB formulation is consistent with the
direct density formulation at the density level.

\begin{figure}[htp]
\centering
\includegraphics[width=0.49\textwidth]{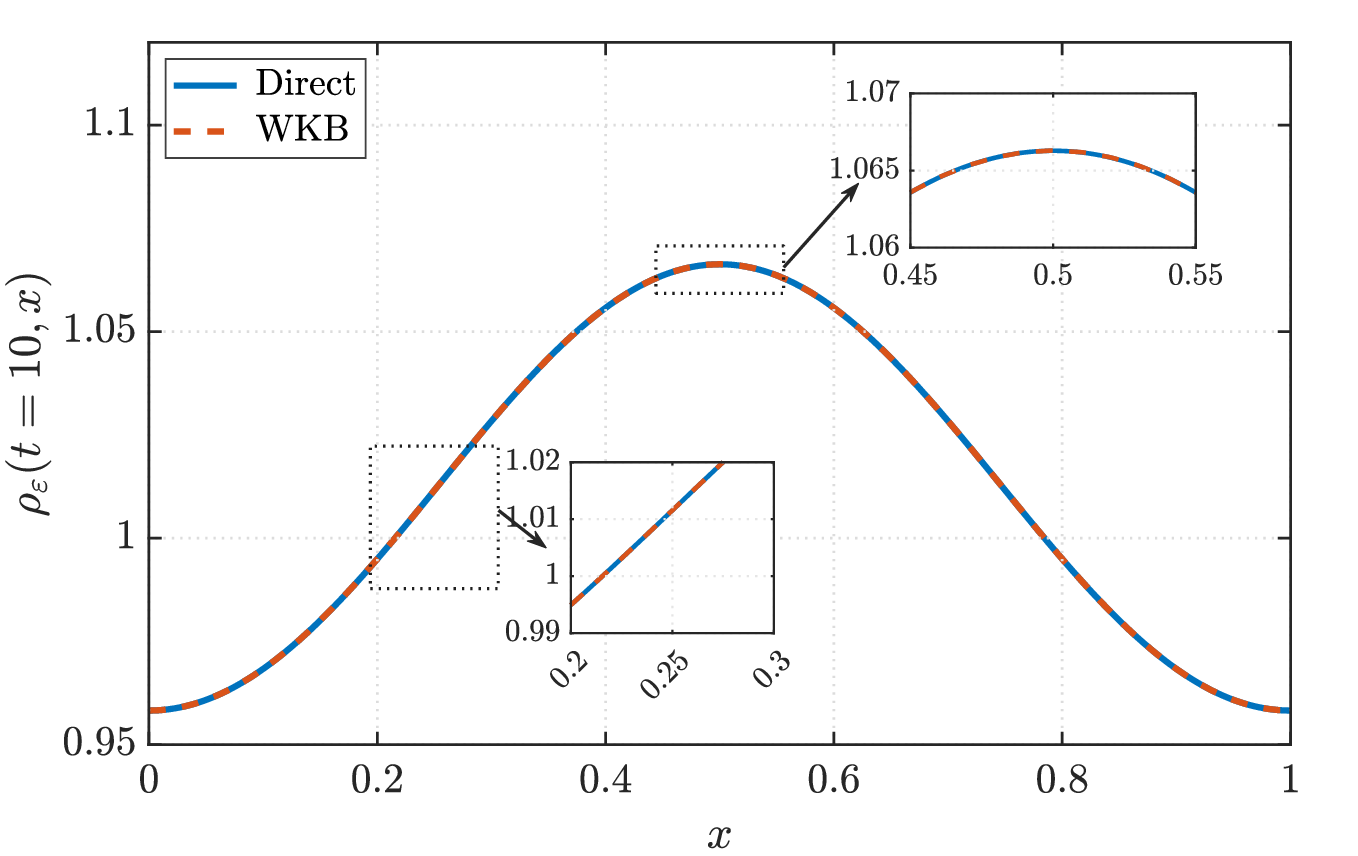}
\includegraphics[width=0.49\textwidth]{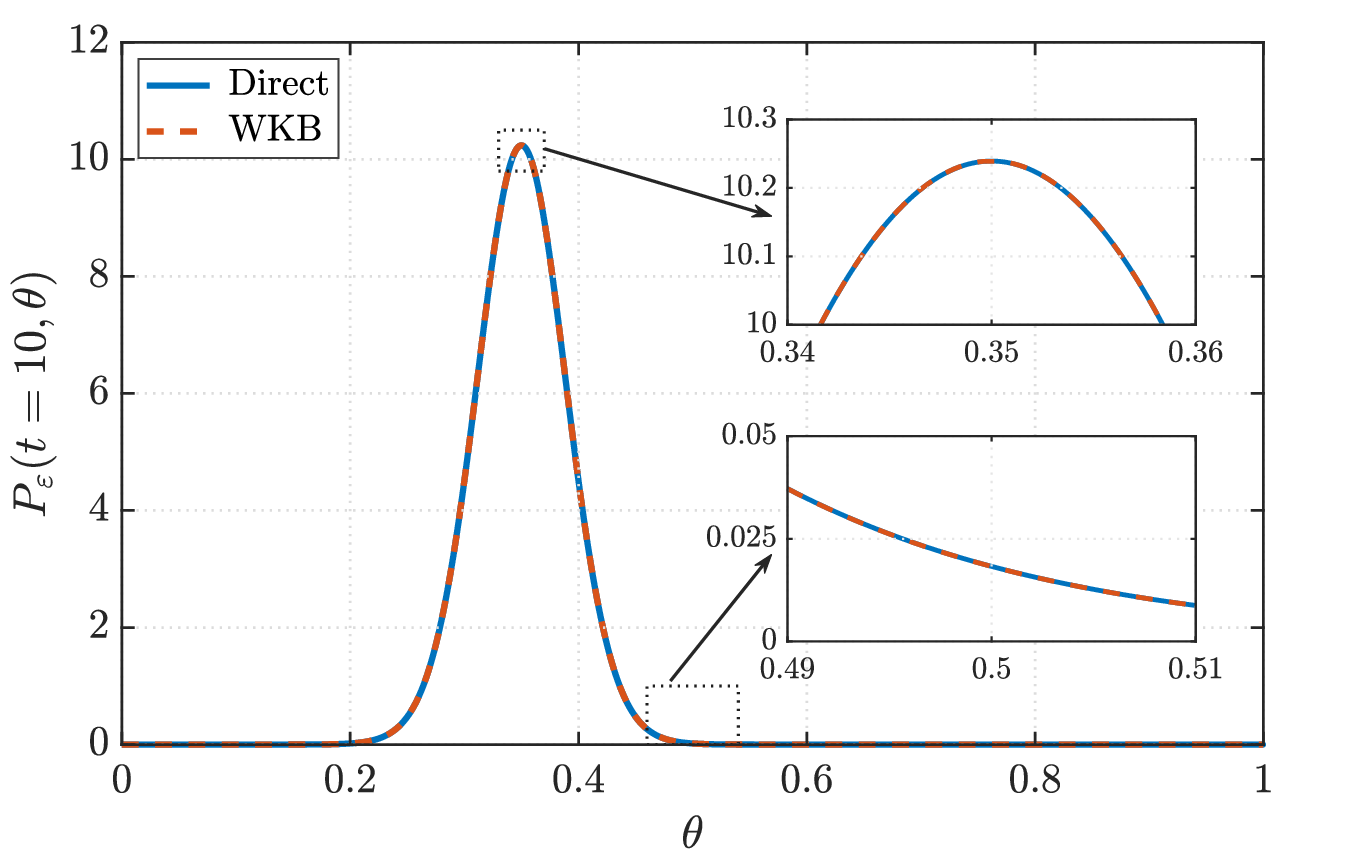}
\caption{Same-grid comparison between the direct density solver and the WKB
formulation.  Left: spatial marginal \(\rho_\varepsilon(t=10, x)\).  Right: trait marginal
\(P_\varepsilon(t=10, \theta)\).  Parameters are \(\varepsilon=10^{-3}\), \(t=10\), \(N_x=100\),
 \(N_\theta=1024\) and $\Delta t = 10^{-3}$.}
\label{fig:rho_pk_same_grid}
\end{figure}

Table~\ref{tab:test3_same_grid_consistency} further reports the same-grid
errors as the trait grid is refined.  The errors decrease steadily with \(N_\theta\).
After the trait grid is moderately refined, the error is reduced by roughly a
factor of four whenever \(N_\theta\) is doubled, which is consistent with the
second-order trait discretization.  This test confirms that the WKB
phase-amplitude formulation does not introduce a density-level inconsistency
when the direct density profile is sufficiently resolved.

\begin{table}[htbp]
\centering
\caption{Same-grid consistency between the direct density solver and the WKB
reconstruction at  \(t=10\) with \(\varepsilon=10^{-3}\), \(N_x=100\) and $\Delta t = 10^{-3}$. Here the errors are measured in the discrete \(\ell^2\)- and
\(\ell^\infty\)-norms over the space--trait grid.}
\label{tab:test3_same_grid_consistency}
\vspace{0.5em}
\begin{tabular}{c|ccccccc}
\hline
 \(N_\theta\)  & 16 & 32 & 64 & 128 & 256 & 512 & 1024\\% & 2048 \\
\hline
 \(\|n_{direct} - n_{_{WKB}}\|_2\)  &1.74E--1 & 3.84E--2 & 8.70E--3  & 2.13E--3  & 5.30E--4 & 1.32E--4 & 3.30E--5\\%  & 8.21E--06  \\
\hline
 \(\|n_{direct} - n_{_{WKB}}\|_\infty\)  &1.42E--1 & 4.59E--2 & 1.03E--2  & 2.63E--3  & 6.54E--4 & 1.63E--4 & 4.08E--5\\%  & 1.02E--05  \\
\hline
\end{tabular}
\end{table}

% ========================= Test 3 ============================
\subsubsection{Long-time concentration in the trait direction}
\label{subsec:test_long_time_concentration_new}
Here we illustrate the concentration process in the trait direction.
According to the rare-mutation asymptotics, the trait marginal
\(P_\varepsilon(t,\theta)\) is expected
to concentrate near the selected trait as time evolves, and the concentration
becomes more pronounced as \(\varepsilon\) decreases.

Figure~\ref{fig:long_time_concentration} shows this behavior.  In the left
panel, \(\varepsilon\) is fixed and the trait marginal is plotted at several
time levels.  
Starting from a broad initial profile, the peak moves toward the
expected selected trait \(\theta_m=0.35\), and the solution approaches a sharply
localized steady profile. In particular, the profiles at the later times are
almost indistinguishable, indicating that a numerical steady state has been
reached.
In the right panel, the final-time marginals are compared for
different values of \(\varepsilon\). The peak becomes narrower and higher as
\(\varepsilon\) decreases, showing the concentration of the trait distribution
toward \(\theta_m\). These results are consistent with the continuous
rare-mutation structure and demonstrate that the WKB formulation captures the
emergence of the selected trait.

\begin{figure}[htp]
\centering
\includegraphics[width=0.48\textwidth]{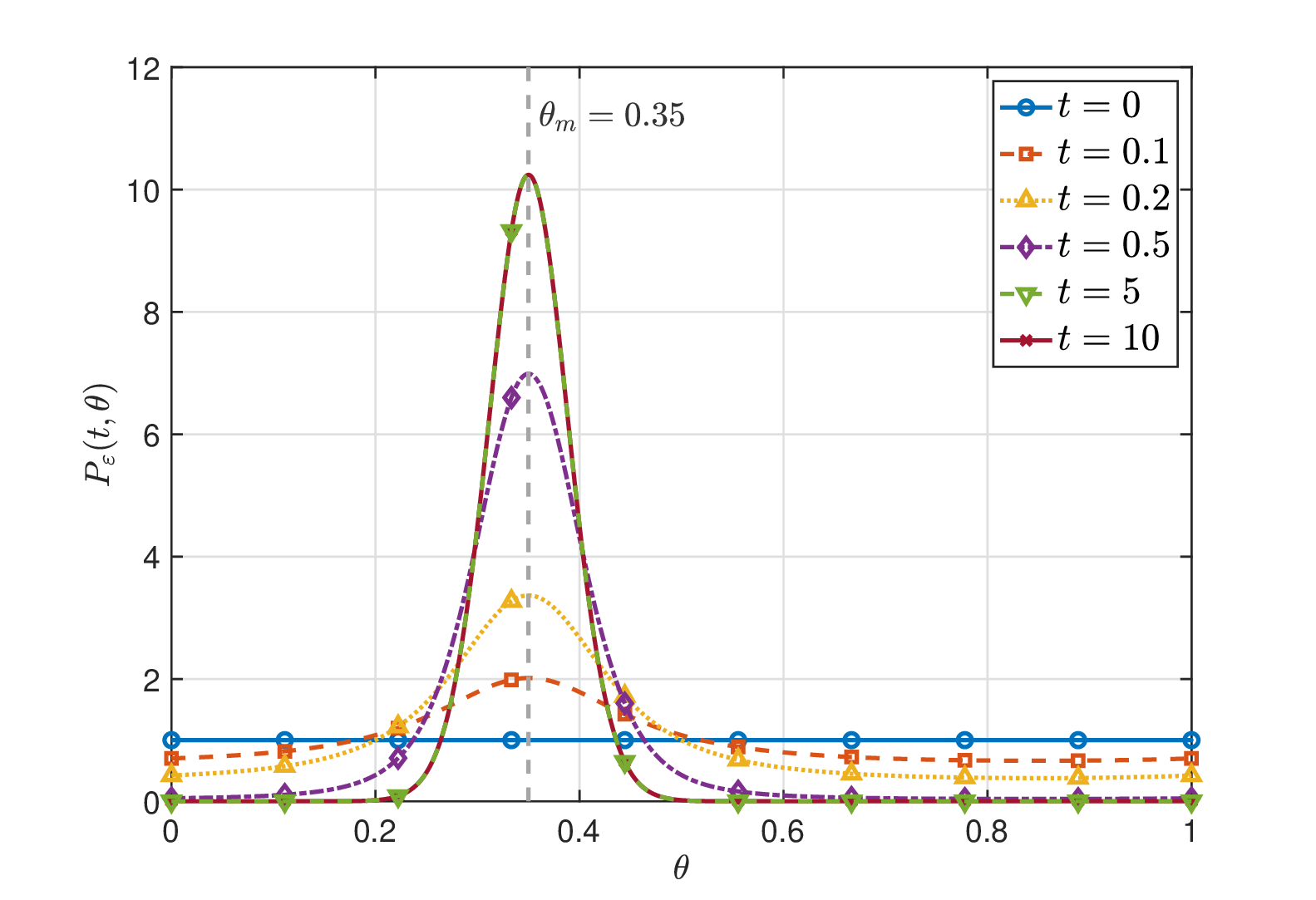}
\includegraphics[width=0.48\textwidth]{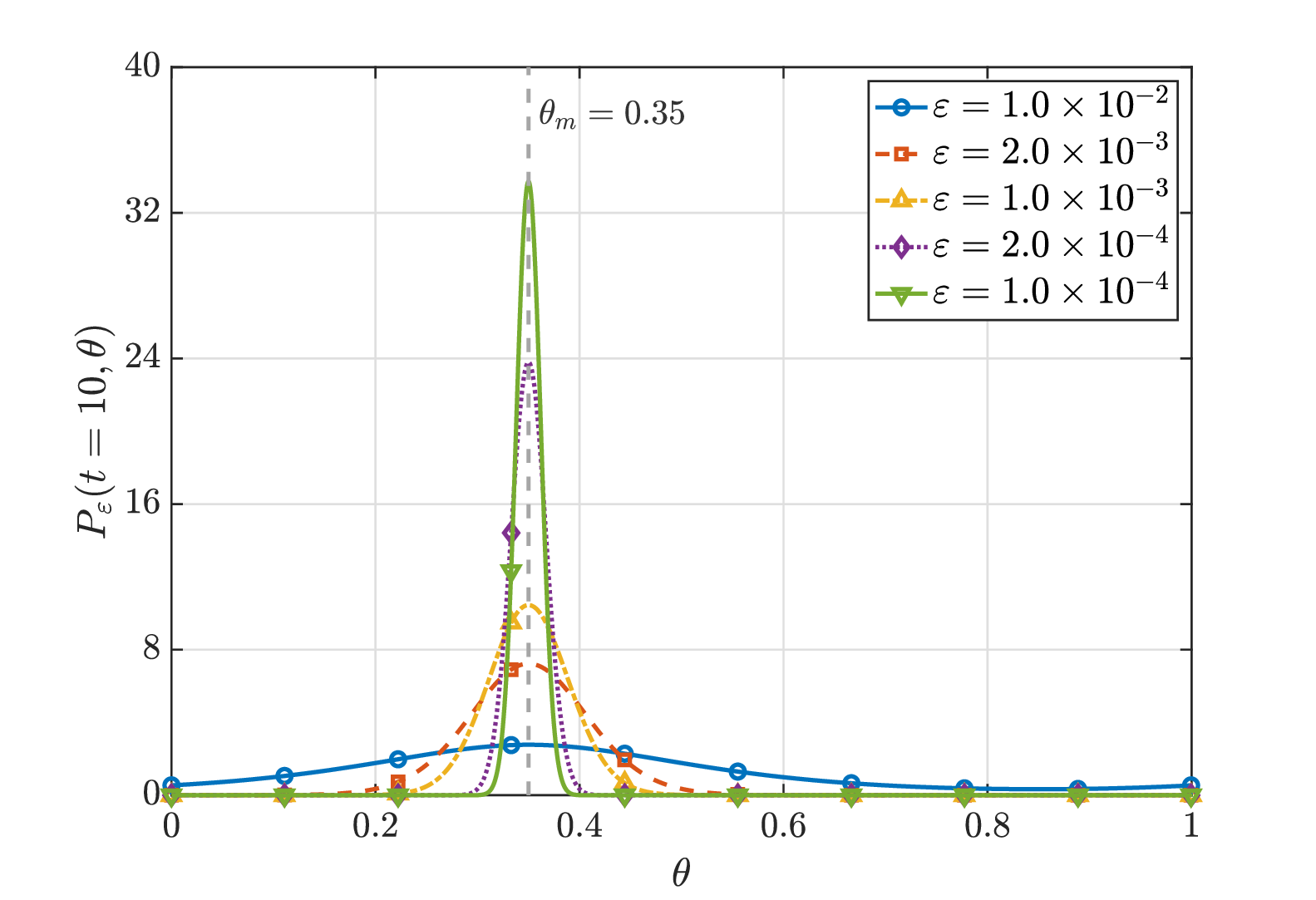}
%\centering
%\includegraphics[width=0.48\textwidth]{long_time_con_nonhomo.eps}
\caption{Long-time concentration in the trait direction. Left: evolution of
\(P_\varepsilon(t,\theta)\) for fixed \(\varepsilon=10^{-3}\), \(N_x=100\),
\(N_\theta=128\), and \(\Delta t=10^{-3}\). Right: final-time trait marginals for
different values of \(\varepsilon\). The vertical dashed line marks the
expected selected trait \(\theta_m=0.35\).}
\label{fig:long_time_concentration}
\end{figure}

% ========================= Test 4 ============================
\subsubsection{WKB versus direct density discretization on coarse trait grids}
\label{subsec:test_wkb_advantage_new}

We next compare the WKB formulation with direct density discretizations
on coarse trait grids, aiming to illustrate the main numerical advantage
of the WKB representation. In the small-mutation regime, the direct
density becomes sharply concentrated in the trait direction, and a coarse
trait grid may fail to resolve the location and shape of the concentration.
By contrast, in the WKB formulation the selected trait is encoded in the
phase variable, which remains much smoother than the density.

For both the plots and the peak diagnostics in this subsection, we use
the same log-based reconstruction of the trait marginal. For a direct
density computation, we  compute
\begin{equation}
\log P_\varepsilon(t,\theta)=\log\left(\int_\Omega n_\varepsilon(t,x,\theta)\,dx \right).
\end{equation}
Similarly, for a WKB computation, we compute
\begin{equation}
\log P_\varepsilon(t,\theta)
=
\frac{u_\varepsilon(t,\theta)}{\varepsilon}
+
\log\left(
\int_\Omega W_\varepsilon(t,x,\theta)\,dx
\right).
\end{equation}
This logarithmic evaluation avoids numerical underflow or overflow when
\(\varepsilon\) is small. When a spatial marginal or a total mass is
computed from the WKB representation, we use a similar
Laplace-type asymptotic approximation to evaluate sums involving
\(\exp(u_\varepsilon(t,\theta)/\varepsilon)\) \cite{WangWong2007, Wong2001}.
The interpolation is then applied to \(\log P_\varepsilon(t,\theta)\), rather than to \(P_\varepsilon(t,\theta)\)
itself.

Figure~\ref{fig:wkb_vs_direct_coarse} shows the comparison for decreasing
values of \(\varepsilon\). The direct density solutions are plotted for
several trait resolutions, while the WKB result is shown as the reference
WKB computation on the same physical problem. 
For moderate
\(\varepsilon\), both approaches give comparable results once the trait
grid is sufficiently refined. As \(\varepsilon\) decreases, however, the
direct density solver becomes much more sensitive to the trait resolution.
In particular, coarse direct grids give visibly shifted or broadened trait
marginals. The WKB formulation remains well aligned with the selected
trait because the phase variable resolves the concentration location
without directly resolving the full density profile.

\begin{figure}[htp]
\centering
\includegraphics[width=0.45\textwidth]{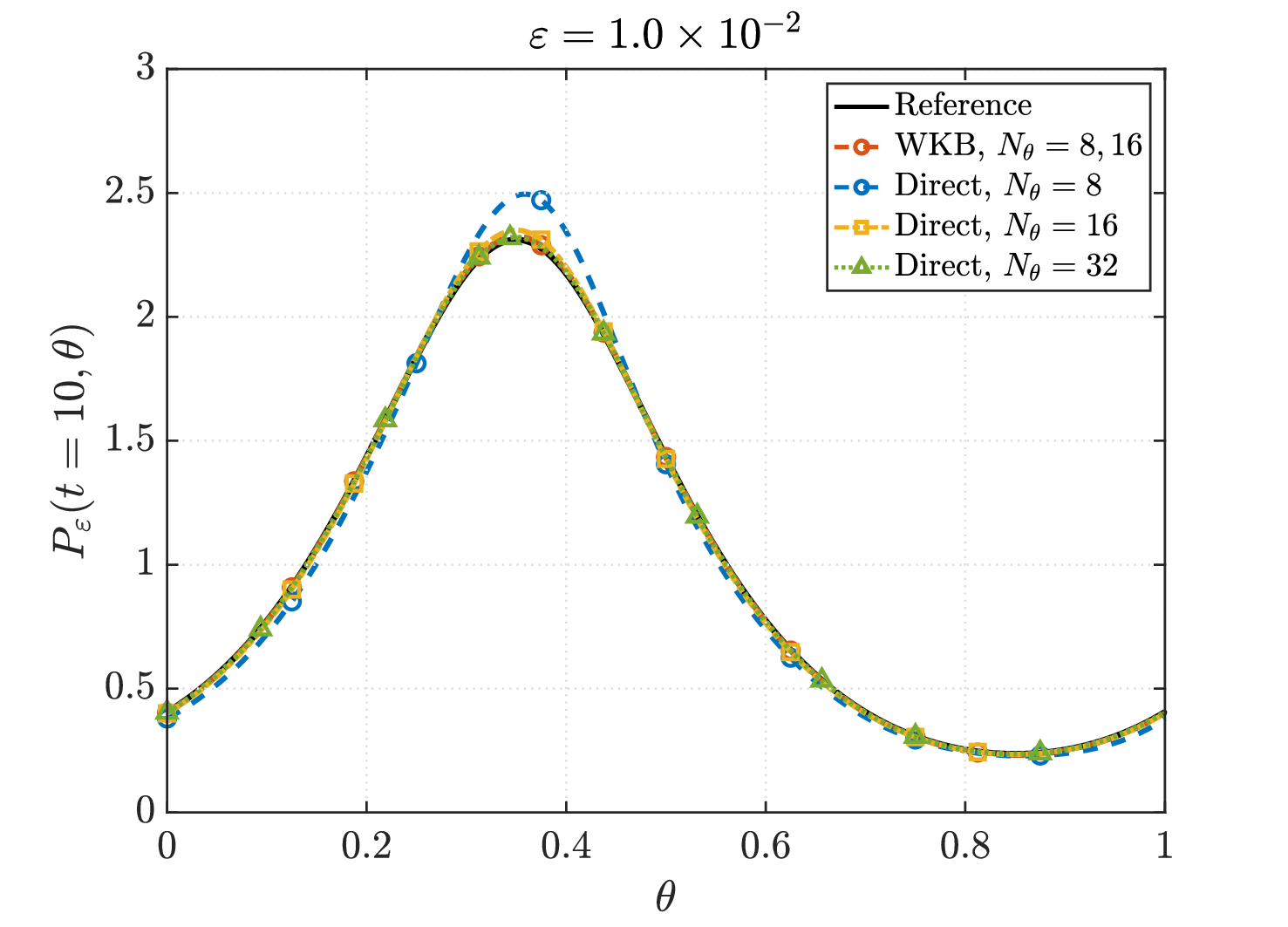}
\includegraphics[width=0.45\textwidth]{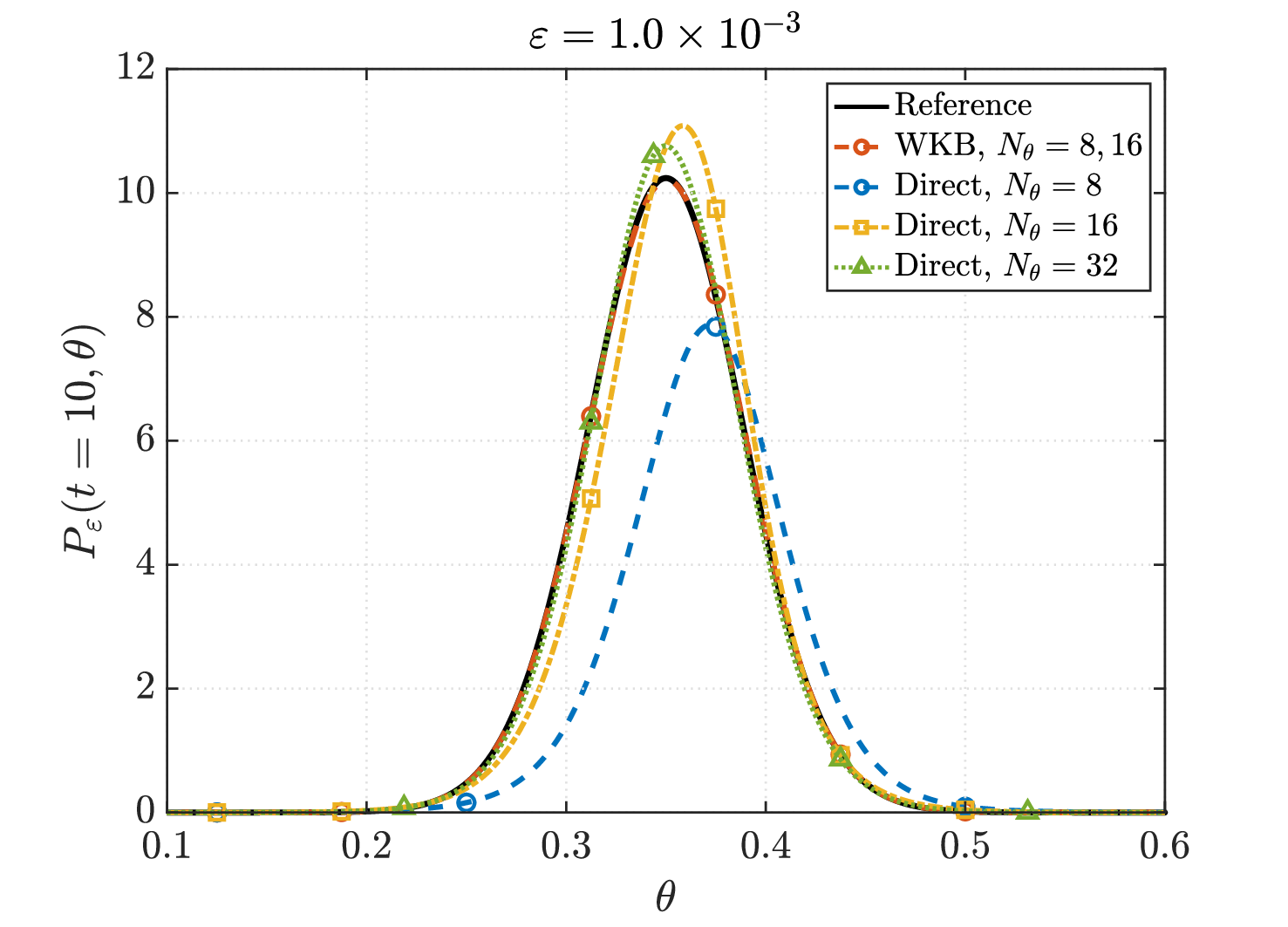}
\centering
\includegraphics[width=0.45\textwidth]{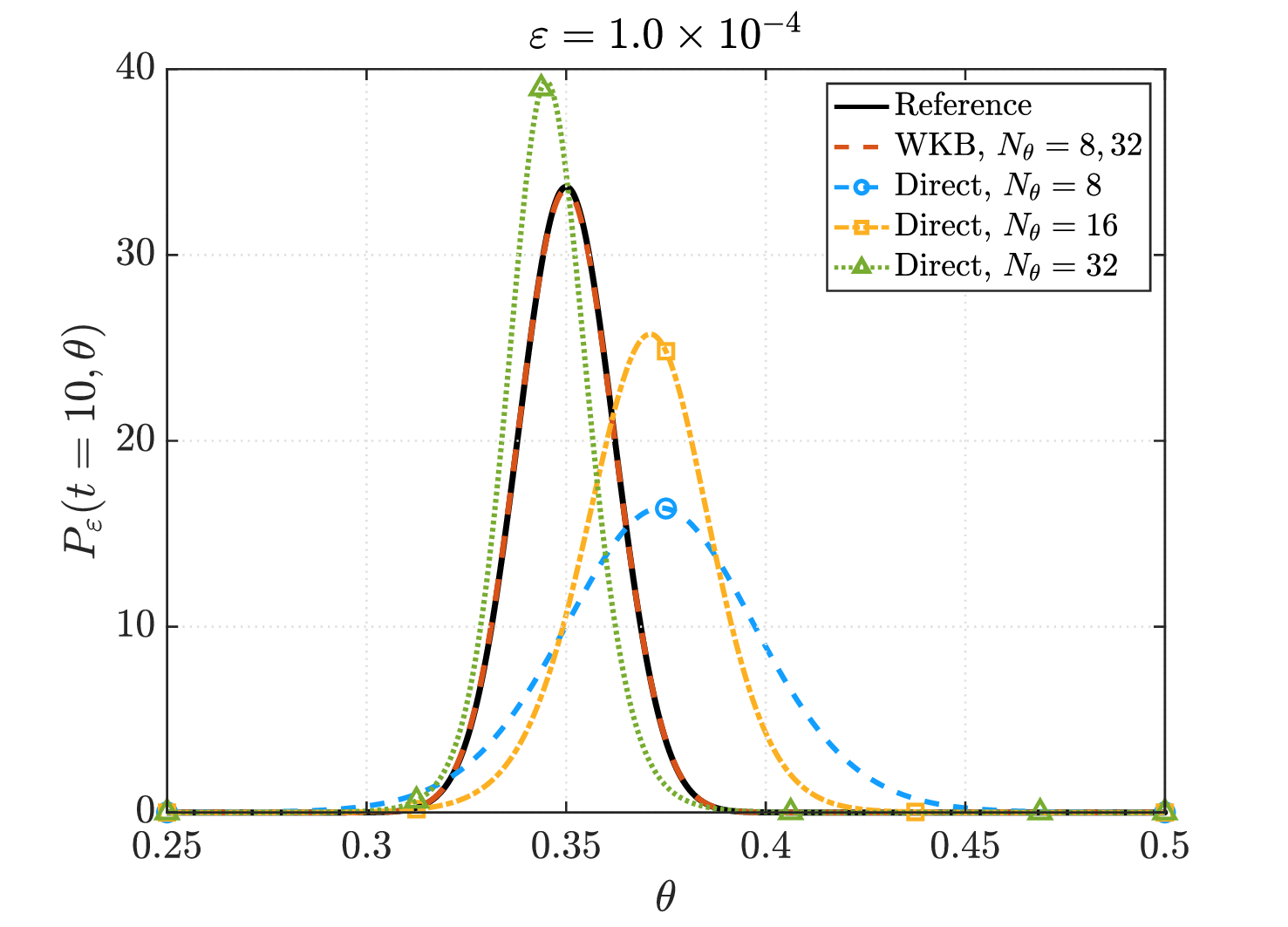}
\includegraphics[width=0.45\textwidth]{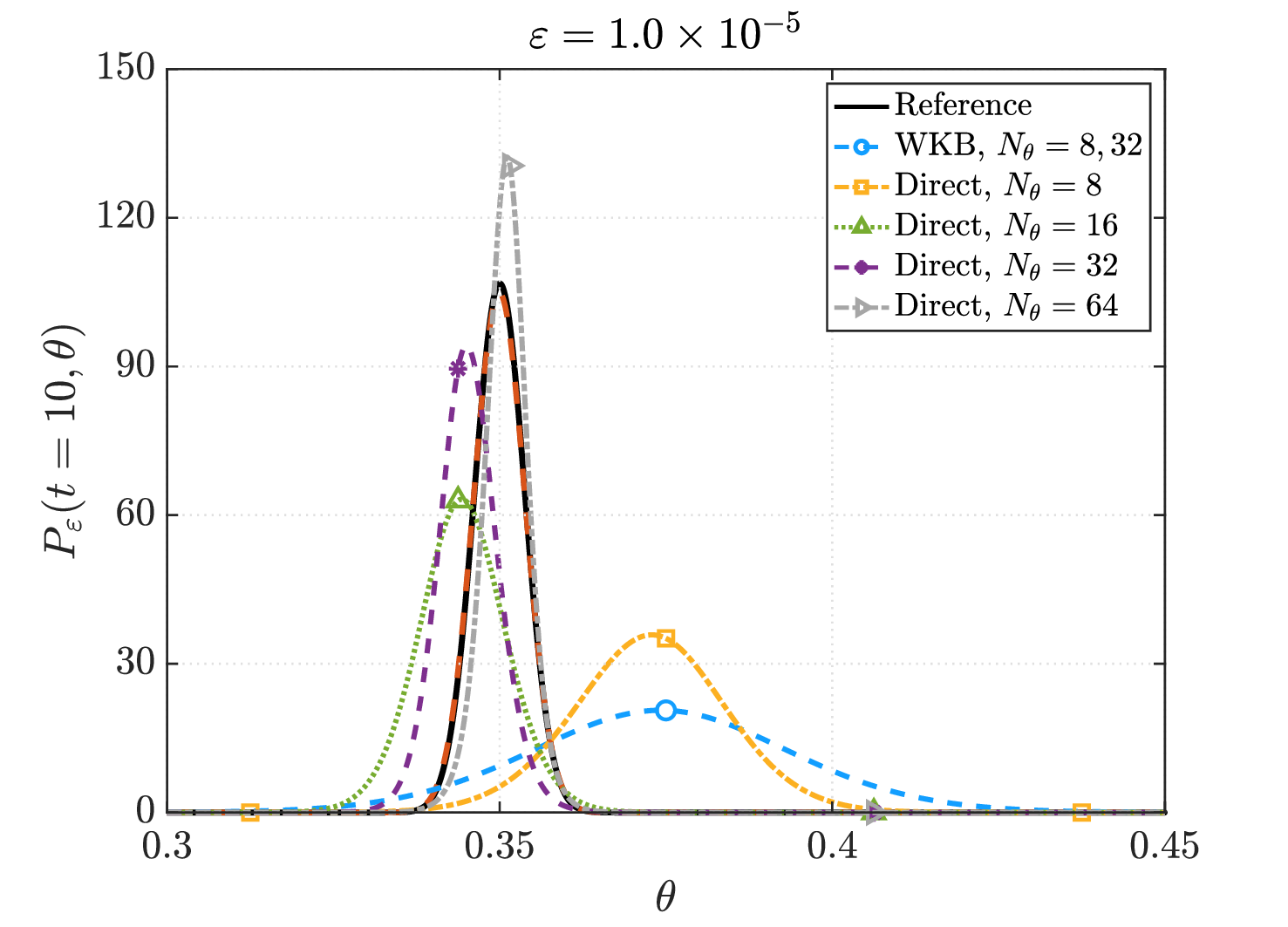}
\caption{Comparison between the WKB formulation and direct density
discretizations on coarse trait grids. The four panels correspond to \(\varepsilon=10^{-2},10^{-3},10^{-4},10^{-5}\), respectively. All trait marginals are reconstructed using the same
log-based postprocessing described in the text. The direct density solver
requires increasingly fine trait grids to resolve the concentrated
marginal, whereas the WKB formulation continues to capture the selected
trait through the phase variable.}
\label{fig:wkb_vs_direct_coarse}
\end{figure}

This experiment shows that the benefit of the WKB formulation becomes more
pronounced as \(\varepsilon\) decreases. The result is consistent with the
asymptotic picture that the density-level trait marginal develops a narrow
concentration, while the WKB phase remains a more regular quantity from
which the selected trait can be identified on a fixed trait grid.

To quantify this comparison, we extract several diagnostics from the reconstructed trait marginal $P_\varepsilon(t,\theta)$, 
including the peak-location error \(e_\theta\), the peak-height error \(e_{\rm PH}\), and the peak-width error \(e_{\rm PW}\), where the peak width is measured by the full width at half maximum.
Here we remark that, in the small \(\varepsilon\) regime, post-processing techniques such as a local quadratic fit of \(\log P_\varepsilon(t,\theta)\) near the peak may be applied to further improve the performance  of the WKB formulation. 

We next examine two sources of computational efficiency gained by the WKB formulation.  
First, the WKB representation reduces the
dominant number of space--trait degrees of freedom needed to resolve the
concentrated marginal.  
We compare direct density discretizations with \(N_\theta^{(n)}\), single-grid WKB computations with \(N_\theta^{(u)}=N_\theta^{(W)}\), and dual-grid WKB computations with \(N_\theta^{(W)}<N_\theta^{(u)}\). Here the superscript specifies which variable the trait grid is associated with. The effective number of unknowns is measured by 
\[ N_{\rm dof}^{\rm dir}=(N_x+1) \cdot N_\theta^{(n)} \]
 for the direct density solver, and by 
 \[ N_{\rm dof}^{\rm WKB}=(N_x+1)\cdot N_\theta^{(W)}+ N_\theta^{(u)} \]
  for the WKB formulation. 
This count reflects the fact that the amplitude \(W\) carries all spatial degrees of freedom, whereas the phase \(u\) is only a one-dimensional trait variable. Using a single-grid WKB solution with \(N_\theta=1024\) as the reference,
Figure~\ref{fig:wkb_accuracy} reports the errors in the selected trait and peak diagnostics against
the effective number of degrees of freedom. The single-grid WKB results converge rapidly as the trait grid is refined. The dual-grid WKB results show that one can keep a relatively fine phase grid while using a much coarser amplitude grid, with only a mild loss of accuracy. In contrast, the direct density solver requires substantially more trait degrees of freedom to resolve the concentrated marginal. This confirms one main computational advantage of the WKB formulation, namely that the concentration location is resolved through the low-dimensional phase, rather than by evolving the full space-trait density on a fine trait grid.

Second,  the WKB representation also allows a larger stable time step in the
small-$\varepsilon$ regime.
In Figure~\ref{fig:time_step_size},  we compare the admissible time steps of the WKB and
direct density solvers over a range of $\varepsilon$ for a
fixed trait resolution.  
The WKB formulation
remains stable for larger time steps, especially when $\varepsilon$ is small,
whereas the direct density solver becomes more restrictive.  
This improvement is closely related to the asymptotic-preserving structure of
the WKB scheme. 
In the rare-mutation regime, the direct solver has to evolve an exponentially
concentrated density profile.  The WKB scheme instead separates the leading
exponential scale and preserves the limiting Hamilton--Jacobi structure at the
discrete level.  Hence its admissible time step is less constrained by the
singular concentration scale of the original density variable.

Overall, the observed
computational gain is not only due to the reduced effective number of unknowns,
but also to the improved time-step robustness of the WKB reformulation in the
rare-mutation regime.

\begin{figure}[htp]
\centering
\includegraphics[width=1.03\textwidth]{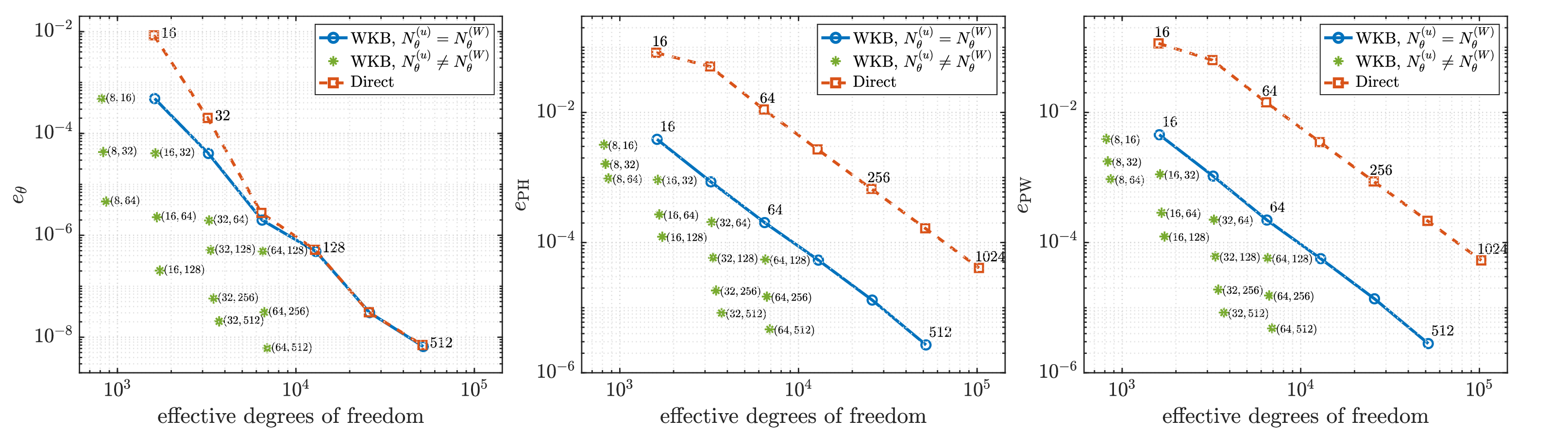}
\caption{Efficiency comparison in terms of effective degrees of freedom.  The
direct density solver uses
$N_{\rm dof}^{\rm dir}=(N_x+1)\cdot N_\theta^{(n)}$, while the WKB formulation uses
$N_{\rm dof}^{\rm WKB}=(N_x+1) \cdot N_\theta^{(W)}+N_\theta^{(u)}$.  The dual-grid WKB
runs use a fine phase grid and a coarser amplitude grid.  The results show that
the WKB formulation can achieve comparable or better peak diagnostics with much fewer
 degrees of freedom.}
\label{fig:wkb_accuracy}
\end{figure}

\begin{figure}[htp]
\centering
\includegraphics[width=0.47\textwidth]{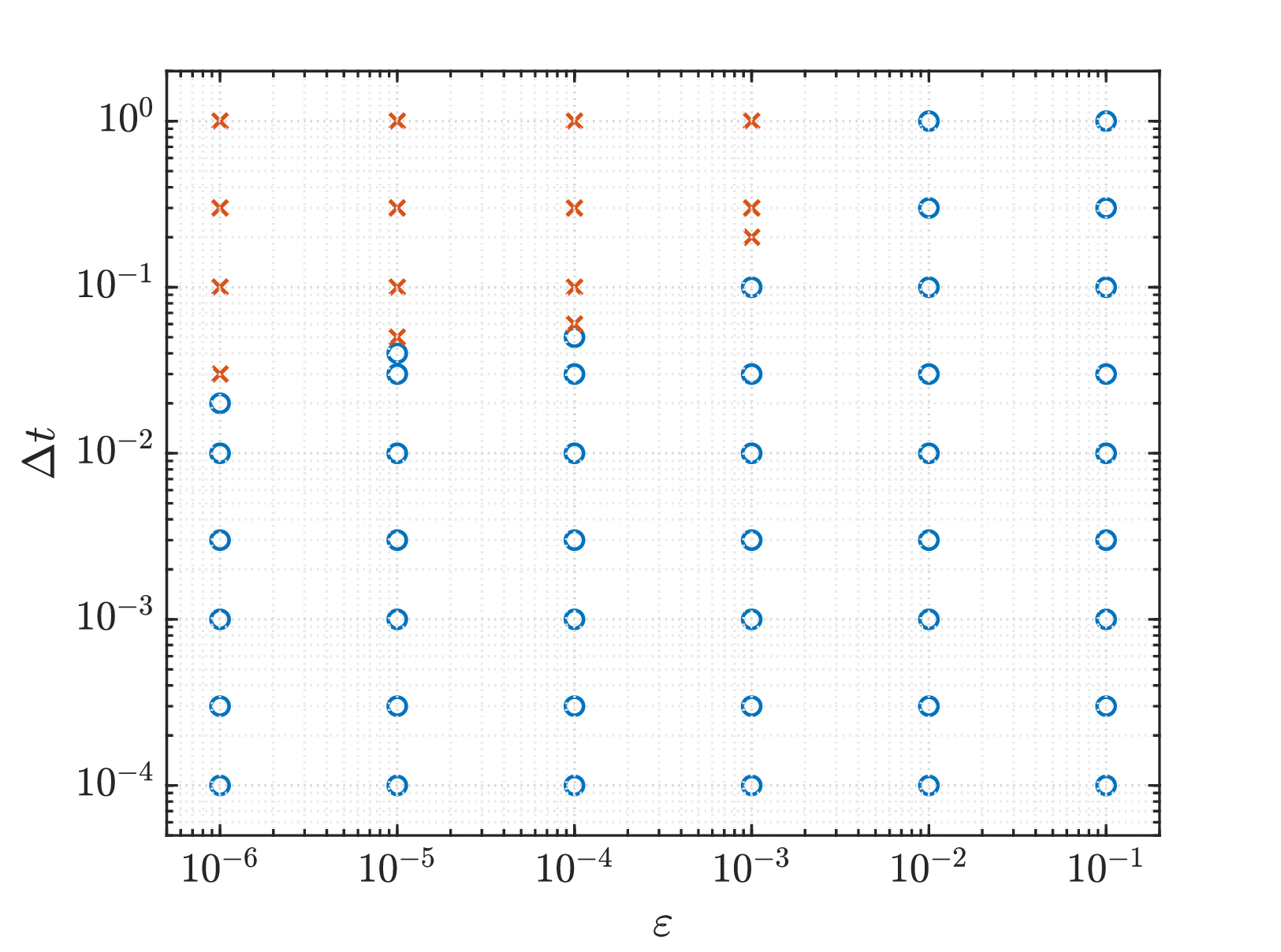}
\includegraphics[width=0.47\textwidth]{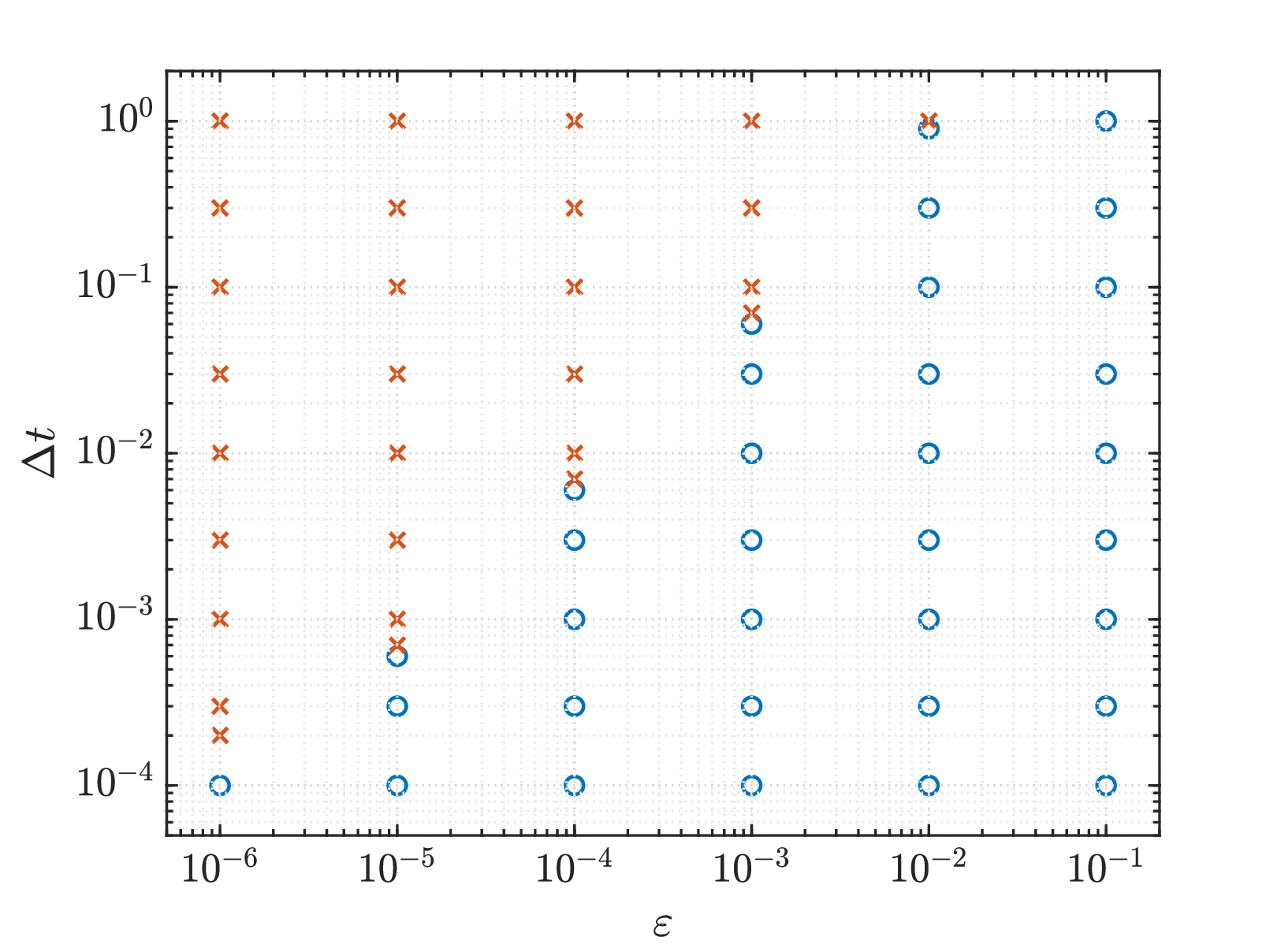}
\caption{Time-step robustness for fixed trait resolution $N_\theta=64$.  Left: WKB
solver.  Right: direct density solver.  Blue circles denote stable computations,
whereas red crosses denote unstable computations.  The WKB
formulation permits larger admissible time steps in the small-$\varepsilon$
regime.}
\label{fig:time_step_size}
\end{figure}

\subsection{Two-dimensional spatial experiment}\label{subsec:test_2d_new}

We finally present a two-dimensional spatial experiment to illustrate that the
WKB formulation can be applied beyond the one-dimensional spatial setting.  The
computational domain is $\Omega=[0,1]^2$, and we choose the spatial growth
rate and the dispersal function as
$$
K(x,y)=K_0+K_1\cos(2\pi x)\cos(2\pi y),
\qquad
D(\theta)=D_{\min}+D_1\bigl(1-\cos(2\pi(\theta-\theta_m))\bigr).
$$
In the computation, we take
$$
K_0=1,\qquad K_1=0.5,\qquad D_{\min}=0.2,\qquad D_1=0.4,
\qquad \theta_m=0.35 .
$$

Figure~\ref{fig:example_2d} shows the numerical result with \(\varepsilon=10^{-3}\), \(N_x=N_y=50\), and \(N_\theta=64\) at \(t=20\). The left panel displays the spatial marginal density $\rho_\varepsilon(t, x,y)$,
which reflects the spatial heterogeneity induced by $K(x,y)$.  The right panel
shows the corresponding trait marginal $P_\varepsilon(t, \theta)$.  The trait distribution is
sharply localized near the expected selected trait $\theta_m=0.35$, indicating
that the WKB computation captures the same rare-mutation concentration
mechanism in the two-dimensional spatial case.

\begin{figure}[htp]
\centering
\includegraphics[width=1.01\textwidth]{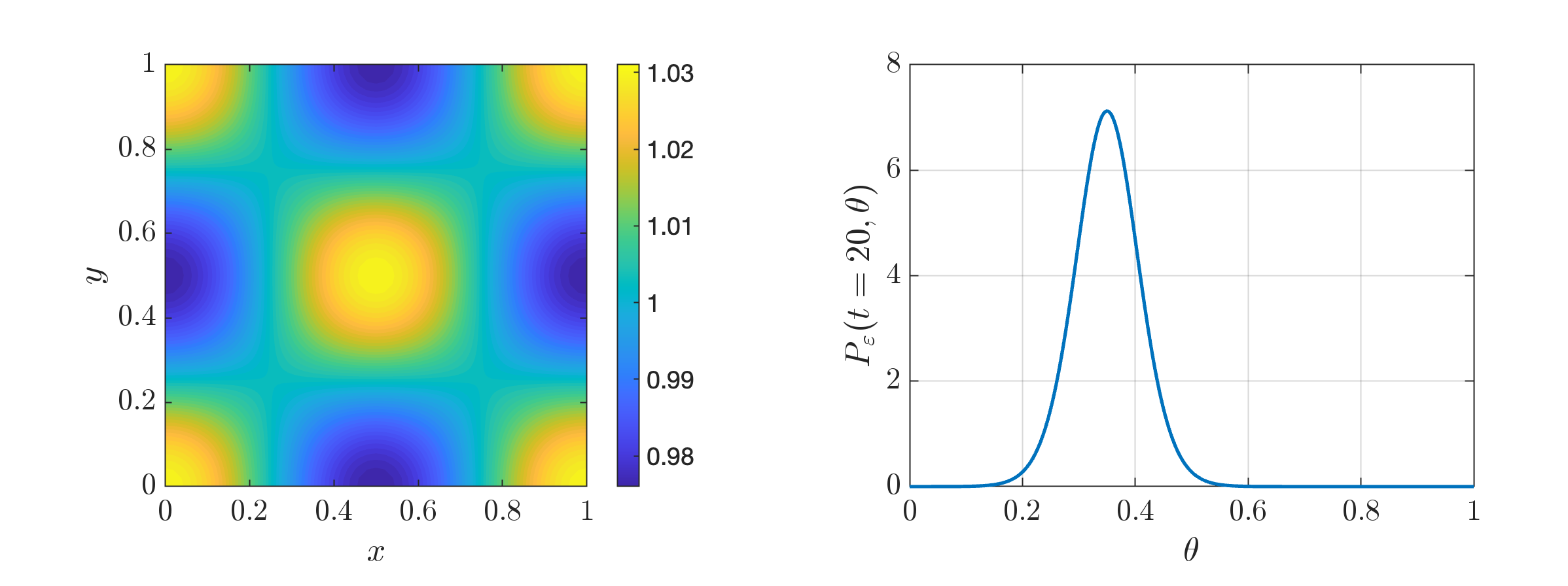}
\caption{Two-dimensional spatial experiment with
\(\varepsilon=10^{-3}\), \(N_x=N_y=50\), and \(N_\theta=64\) at \(t=20\).
Left: spatial marginal density \(\rho_\varepsilon(t,x,y)\).
Right: trait marginal \(P_\varepsilon(t,\theta)\).}
\label{fig:example_2d}
\end{figure}

%=============================================================
\section{Conclusion} \label{sec:conclusion}

In this paper, we developed a WKB-based fixed-grid method for a dispersal
evolution model in the rare-mutation regime.  
We proved positivity preservation and established a stationary
fixed-grid asymptotic-preserving structure under a one-sided weighted remainder
condition.  
Numerical experiments show that the method accurately recovers the
selected trait on coarse trait grids, reduces the dominant space--trait degrees
of freedom through a dual-grid implementation, and allows larger stable time
steps for small $\varepsilon$.  A two-dimensional test further illustrates its
ability to capture spatial heterogeneity and trait concentration beyond the
one-dimensional setting.

\appendix
\setcounter{equation}{0}
\renewcommand{\theequation}{\Alph{section}.\arabic{equation}}

\section{Numerical Hamiltonians and conservative fluxes}
\label{app:hamiltonian_flux}

For the monotone theory in Section~3, the numerical Hamiltonian
\(\widehat H(a,b)\) is assumed to be consistent with \(-p^2\), nondecreasing in
its first argument, and nonincreasing in its second argument. For any real number \(r\), let
\[
    r_+ := \max\{r,0\},\qquad r_- := \min\{r,0\}.
\]
One typical 
choice is the Godunov Hamiltonian
\begin{equation}
\label{eq:app_Godunov_H}
    \widehat H_G(a,b)
    =
    -(a_-)^2-(b_+)^2.
\end{equation}
In practice, high-order one-sided WENO reconstructions can be applied \cite{JiangPeng, OsherShu, ShuHJ}.  
This is a numerical option for
improving resolution of the phase profile.  

The conservative flux in the equation of $W$ can be chosen as 
\[
    D_\theta\widehat F(W,u)_{j,k}
    =
    \frac{\widehat F_{j,k+1/2}-\widehat F_{j,k-1/2}}{\Delta\theta},
    \qquad
    a_{k+1/2}
    =
    -\frac{u_{k+1}-u_k}{\Delta\theta},
\]
with the upwind flux 
\begin{equation}
\label{eq:app_upwind_flux}
    \widehat F^{\rm up}_{j,k+1/2}
    =
    (a_{k+1/2})_+ W_{j,k}
    +
    (a_{k+1/2})_- W_{j,k+1}.
\end{equation}
%A Lax--Friedrichs flux is
%\begin{equation}
%\label{eq:app_LF_flux}
%    \widehat F^{\rm LF}_{j,k+1/2}
%    =
%    \frac12 a_{k+1/2}(W_{j,k}+W_{j,k+1})
%    -
%    \frac12\alpha_{k+1/2}(W_{j,k+1}-W_{j,k}),
%    \qquad
%    \alpha_{k+1/2}\ge |a_{k+1/2}|.
%\end{equation}
%Both fluxes are conservative in the periodic trait variable.  The upwind flux
%also satisfies the quasi-positivity condition used in
%Lemma~\ref{lem:positivity_W}.

\section{Remainder bound for a conservative density discretization}
\label{app:remainder}

We give a simple situation in which Assumption~\ref{ass:remainder_stability} can be
verified.  Suppose that the reconstructed density satisfies the conservative
semi-discrete density equation
\begin{equation}
\label{eq:app_density_compatible_n}
    \varepsilon\frac{\mathrm d}{\mathrm dt}n_{j,k}^\varepsilon
    -
    D_k\delta_x^2 n_{j,k}^\varepsilon
    -
    \varepsilon^2\delta_\theta^2 n_{j,k}^\varepsilon
    =
    \bigl(K_j-\rho_{j}^\varepsilon\bigr)n_{j,k}^\varepsilon .
\end{equation}
Then periodic summation by parts gives
\[
    \mathcal R_j^{\varepsilon,D}
    =
    \varepsilon^2\Delta\theta
    \sum_k
    n_{j,k}^\varepsilon
    \delta_\theta^2\left(\frac{1}{D_k}\right).
\]
In this case, 
\[
    \bigl(R_j^{\varepsilon,D}\bigr)_+
    \le
    \varepsilon^2 C_{D,h}\rho_j^\varepsilon,
    \qquad
    C_{D,h}
    =
    \max_k
    \left|
    \delta_\theta^2\left(\frac{1}{D_k}\right)
    \right|.
\]
This verifies Assumption~\ref{ass:remainder_stability} with \(C_{R,h}=C_{D,h}\) for $\varepsilon\in(0,1]$ and
\(r_h=0\).

% =============================================================================
%                                                            ACKNOWLEDGEMENTS
% =============================================================================
\bigskip
\noindent{\large\bf Acknowledgements.}  \ 
X. Ruan acknowledges support  from the National Natural Science Foundation of China under grant  12201436 and the R\&D Program of Beijing Municipal Education Commission under grant KM202310028016.
W. Huang acknowledges support  from the National Natural Science Foundation of China under grant  12001034.

\end{document}